\documentclass[a4paper,11pt]{amsart}
\usepackage[utf8x]{inputenc}
\usepackage{amsthm}
\usepackage{amssymb}
\usepackage{amsmath}
\usepackage{verbatim}
\usepackage{url}
\usepackage{mathtools}
\usepackage{geometry}
\usepackage{tikz}
\usepackage{latexsym}
\usepackage{stmaryrd}
\usepackage{afterpage}
\usepackage{youngtab}
\usepackage{psfrag}

\usetikzlibrary{chains}

\makeatletter
\def\Ddots{\mathinner{\mkern1mu\raise\p@
\vbox{\kern7\p@\hbox{.}}\mkern2mu
\raise4\p@\hbox{.}\mkern2mu\raise7\p@\hbox{.}\mkern1mu}}
\makeatother

\theoremstyle{plain}
\newtheorem{theo}{Theorem}[section]
\newtheorem{prop}[theo]{Proposition}
\newtheorem{cor}[theo]{Corollary}
\newtheorem{lemma}[theo]{Lemma}

\theoremstyle{definition}
\newtheorem{defn}[theo]{Definition}

\theoremstyle{remark}
\newtheorem*{remark}{Remark}

\newcommand{\Wlus}{\mathcal{W}_{\mathrm{Lus}}}
\newcommand{\Xlus}{\check{X}_{\mathrm{Lus}}}
\newcommand{\Xlie}{\check{X}_{\mathrm{Lie}}}
\newcommand{\Xcan}{\check{X}_{\mathrm{can}}}
\newcommand{\Xgiv}{\check{X}_{\mathrm{Giv}}}

\newcommand{\Wlie}{\mathcal{W}_{\mathrm{Lie}}}
\newcommand{\Wlieq}{\mathcal{W}_{\mathrm{Lie},q}}
\newcommand{\Wcan}{\mathcal{W}_{\mathrm{can}}}
\newcommand{\Wcanq}{\mathcal{W}_{\mathrm{can},q}}
\newcommand{\Wgiv}{\mathcal{W}_{\mathrm{Giv}}}

\newcommand{\x}{\times}
\newcommand{\Q}{Q}
\newcommand{\p}{{p}}

\newcommand{\inv}{^{-1}}

\newcommand{\Sp}{\mathrm{Sp}}

\renewcommand{\P}{\mathbb{P}}

\newcommand{\C}{\mathbb{C}}
\newcommand{\Z}{\mathbb{Z}}

\newcommand{\HH}{\mathcal{H}}

\newcommand{\qH}{\operatorname{qH}}

\newcommand{\Gr}{\operatorname{Gr}}

\newcommand{\Proj}{\operatorname{Proj}}

 \newcommand{\XGS}{\check{X}_{\mathrm{GS}}}
\newcommand{\WGS}{\mathcal{W}_{\mathrm{GS}}}
\newcommand{\Xprz}{\check{X}_{\mathrm{Prz}}}
\newcommand{\Wlusq}{\mathcal{W}_{\mathrm{Lus},q}}
\newcommand{\Wprz}{\mathcal{W}_{\mathrm{Prz}}}
\newcommand{\Wprzq}{\mathcal{W}_{\mathrm{Prz},q}}

\newcommand{\Wtheta}{\mathcal W_\theta}
\newcommand{\Ftheta}{\mathcal F_\theta}
\renewcommand{\Wlusq}{\mathcal{W}_{\mathrm{Lus},q}}

\newcommand{\pilie}{\pi_{\mathrm{Lie}}}
\newcommand{\X}{\check{\mathbb X}_{\operatorname{can}}}

\newcommand{\Xcheckcan}{\check{\mathbb X}_{\operatorname{can}}}
\newcommand{\XcheckLie}{\check{\mathbb X}_{\operatorname{Lie}}}

\newcommand{\denom}{\delta}

\newcommand{\SO}{\operatorname{SO}}

\dedicatory{This paper is dedicated to George Lusztig on his 70$^{\sl{th}}$ birthday.}



\title{A comparison of Landau-Ginzburg models for odd dimensional quadrics}
\author{Clelia Pech}
\address{University of Kent, 
Canterbury,  
CT2 7NZ,}

\author{Konstanze Rietsch}
\address{King's College London, Strand, London, WC2R 2LS}

\thanks{This work is supported by the Leverhulme Trust grant  n\textsuperscript{o} F07040AW - A Lie theoretic approach to derived categories of flag varieties $G/P$}

\begin{document}

\begin{abstract}
	In \cite{rietsch}, the second author defined a Landau-Ginzburg model for homogeneous spaces $G/P$. In this paper, we reformulate this LG model in the case of the odd-dimensional quadric $X=\Q_{2m-1}$. Namely we introduce a regular function $\Wcan$ on a variety $\Xcan\x\C^*$, where $\Xcan$ is the complement of a particular anticanonical divisor in the projective space $\C\P^{2m-1}=\P(H^*(X,\C)^*)$. Firstly we prove that the Jacobi ring associated to $\Wcan$ is isomorphic to the quantum cohomology ring of the quadric, and that this isomorphism is compatible with the identification of homogeneous coordinates on  
$\Xcan\subset \C\P^{2m-1}$ with elements of $H^*(X,\C)$. Secondly we find a very natural Laurent polynomial formula for $\Wcan$ by restricting it to a `Lusztig torus' in $\Xcan$. Thirdly we show that the Dubrovin connection on $H^*(X,\C[q])$ embeds into the Gauss-Manin system associated to $\Wcan$ and deduce a flat section formula in terms of oscillating integrals. Finally, we compare $(\Xcan,\Wcan)$ with previous Landau-Ginzburg models defined for odd quadrics. Namely, we prove that it is a partial compactification of Givental's original LG model \cite{Givental:EqGW}. We show that our LG model is isomorphic to the Lie-theoretic LG model from \cite{rietsch}. Moreover it is  birationally equivalent to an LG model introduced by Gorbounov and Smirnov \cite{GS}, and it is algebraically isomorphic to Gorbounov and Smirnov's mirror for $Q_3$, implying a tameness property in that case.   
\end{abstract}

\maketitle

\section{Introduction}

The geometric Satake correspondence  \cite{lusztig,Gin:GS,MV} constructs representations of a reductive algebraic group $G$ in terms of geometry of the affine Grassmannian of the Langlands dual group $G^\vee$. It has its origins in the seminal paper of Lusztig \cite{lusztig}. In this paper we describe the mirror symmetry partner of a smooth, odd-dimensional complex quadric $X=\Q_{2m-1}$ from the point of view of its automorphism group $G^\vee=\SO_{2m+1}(\C)$, Langlands duality and the geometric Satake correspondence. 

Recall that the Langlands dual group of $\SO_{2m+1}(\C)$ is the symplectic group $\Sp_{2m}(\C)$. The geometric Satake correspondence provides us with a `Langlands dual' interpretation of the cohomology $H^*(X,\C)$ of the smooth quadric $X=Q_{2m-1}$ as follows. The quadric $X$ appears as one of the simplest Schubert varieties inside the affine Grassmannian of $\SO_{2m+1}(\C)$,
\[
X\hookrightarrow \Gr_{G^\vee}=\SO_{2m+1}(\C((t)))/\SO_{2m+1}(\C[[t]]).
\] 
Namely this Schubert variety is associated to the first fundamental coweight of $\SO_{2m+1}(\C)$.  The geometric Satake correspondence reinterprets this coweight as a dominant {\it weight} for the Langlands dual group, $\Sp_{2m}(\C)$. Moreover the intersection cohomology of the associated Schubert variety $X$ is then understood to be the representation of $\Sp_{2m}(\C)$ with that highest weight. In our setting, since the quadric is smooth, the intersection cohomology coincides with the usual  cohomology of $X$ and we obtain the interpretation,
\[
H^*(X,\C)=\C^{2m}=V_{\omega_1},
\]
of the cohomology of $X$, where $V_{\omega_1}$ is the defining representation of $G=\Sp_{2m}(\C)$.

In mirror symmetry a `mirror dual' construction of the {\it quantum} cohomology ring of $X$ is sought, along with other structures involving Gromov-Witten invariants of $X$, see \cite{BeauvilleQCoh, CoxKatz}. In the setting of the quadric $X$, the (small) quantum cohomology ring is a $1$-parameter deformation of $H^*(X,\C)$, whose structure constants are Gromov-Witten invariants that count $3$-pointed genus $0$ holomorphic curves in $X$ subject to certain constraints. The result is a commutative algebra structure on $H^*(X,\C)\otimes \C[q]$ which recovers the usual cohomology ring when $q\to 0$. For explicit formulas in the case of quadrics we refer to \cite{CMP:II}.  

There are various previous mirror constructions that apply to odd quadrics, which we recall in Section~\ref{s:LGoverview}. Some of these already recover the quantum cohomology ring, and one construction is already in terms of the Langlands dual group $\Sp_{2n}(\C)$. The main new construction we introduce in this paper combines the geometric Satake correspondence  with the `Langlands dual group' construction of the mirror. As a result we construct a mirror for the quadric $X$ that is expressed in terms of coordinates which are naturally identified with cohomology classes of the quadric. An analogous construction was carried out for Grassmannians in \cite{MR}, and then for Lagrangian Grassmannians in~\cite{PeRi}.    

  The mirror of the quadric $X$ takes the form of a \emph{Landau-Ginzburg model} or \emph{LG model}, that is, of a pair $(\Xcan,\Wcan)$, where $\Xcan$ is an affine Calabi-Yau variety and $\Wcan$ is a regular function $\Xcan\to \C$.  In our construction $\Xcan$ is the complement  of a particular anticanonical divisor in projective space,
  \[
\Xcheckcan:= \C\P^{2m-1}=\P(H^*(X,\C)^*)=\Proj\left(\C[\p_0,\dotsc,\p_{2m-1}]\right).
  \] 
  Here $\C\P^{2m-1}$ is viewed as the homogeneous space $\P(V_{\omega_1}^*)$ for the symplectic group $G=\Sp_{2m}(\C)$. Thus the first equality is by the geometric Satake correspondence. In the second equality, the variables or homogeneous coordinates,  $\p_0,\p_1,\dotsc, \p_{2m-1}$ are identified with the Schubert basis,  $\sigma_0,\sigma_1,\dotsc, \sigma_{2m-1}\in H^*(X,\C)$, which has one element  in each even degree. 
  
To give a concrete example, in the case of $X=\Q_3$ our formula reads
\begin{equation*}
\Wcanq = \frac{\p_1}{\p_0}+\frac{\p_2^2}{\p_1 \p_2-\p_0\p_3}+ q \frac{\p_1}{\p_3}
\end{equation*}
in terms of the homogeneous coordinates $(\p_0:\p_1:\p_2:\p_3)$ on $\Xcheckcan=\C\P^3$, which are identified with the Schubert classes of $\Q_3$. 

In the case of $X=Q_5$ our formula reads
\begin{equation*}
\Wcanq = \frac{\p_1}{\p_0}+\frac{\p_2p_4}{\p_1 \p_4-\p_0\p_5}+ \frac{\p_3^2}{\p_2 \p_3-\p_1\p_4+\p_0 \p_5} + q \frac{\p_1}{\p_5}
\end{equation*}
in terms of the homogeneous coordinates $(\p_0:\p_1:\p_2:\p_3:\p_4:\p_5)$ on $\Xcheckcan=\C\P^5$, which are identified with the Schubert classes of $\Q_5$.

We compare our formula with previous constructions of Landau-Ginzburg models \cite{Givental:EqGW,rietsch,GS}, and obtain various mirror theorems for our LG model.

\subsection{Quantum Cohomology} \label{s:QCoh1}
The LG model $(\Xcan,\Wcan)$ provides the following Jacobi ring description of the small quantum cohomology ring of the quadric~$X$. Let $\denom_{i}$ denote the $i^{\rm th}$ quadratic denominator of the superpotential $\Wcan$ when $i=1\dotsc m-1$, and $\denom_m=\p_{2m-1}$. Then   
\begin{multline}\label{e:qcoh}
\qH^*(X,\C)[q\inv]=\C[\Xcan\x \C^*_q]/\left(\frac{\partial \Wcanq}{\partial \p_1}, \frac{\partial \Wcanq}{\partial \p_2}, \dotsc, \frac{\partial \Wcanq}{\partial \p_{2m-1}}\right)\\
=\C[\p_1,\dotsc, \p_{2m-1},\denom_1\inv, \dotsc, \denom_{m}\inv,q^{\pm 1}]/\left(\frac{\partial \Wcanq}{\partial \p_1}, \frac{\partial \Wcanq}{\partial \p_2}, \dotsc, \frac{\partial \Wcanq}{\partial \p_{2m-1}}\right),
\end{multline}
where $p_i$ is identified with the (unique) Schubert class generator $\sigma_i$ in $\qH^{2i}(X,\C)$ and we have set $\p_0=1$. In particular, the equality \eqref{e:qcoh} says that the $p_i$ span the right hand side as a free $\C[q,q\inv]$-module, and the multiplicative structure constants compute $3$-point genus~$0$ Gromov-Witten invariants of the quadric $X$.

\subsection{The Dubrovin connection and flat sections}
The next more sophisticated mirror theorem says that the Gauss-Manin connection defined using $\Wcanq$ recovers Dubrovin's connection  on the free $\C[q^{\pm 1}]$-module, $H^*(X,\C[q^{\pm 1}])$, which is defined using the small quantum cup product $\star_q$, see \cite{Dub:2DTFT,CoxKatz}. The precise statement is formulated in Section~\ref{s:consequences}.  Namely we have a natural embedding of the Dubrovin connection into the Gauss-Manin system, Theorem~\ref{t:connections}, in which the coordinates $\p_i$ on the Gauss-Manin side match up with the Schubert classes $\sigma_i$ on the Dubrovin connection side. The theorem in particular implies an integral formula for a global flat section of the connection, which is stated in Corollary~\ref{c:int}. Much earlier Givental \cite{Givental:EqGW} constructed a flat section for this connection without using mirror symmetry, as a power series with coefficients given by descendent $2$-point Gromov-Witten invariants of~$X$. Our Corollary~\ref{c:int} implies integral formulas for these invariants, via a comparison with Givental's formula. See the sequel paper~\cite{PRW:quadrics}.    

To illustrate the flat section formulas, let $X=Q_3$ and $\omega$ be a meromorphic $3$-form on $\C\P^3$ with simple poles along the divisor 
\[D=\{p_0=0\}\cup\{p_3=0\}\cup\{\p_1 \p_2-\p_0\p_3=0\}.\]
Suppose $\Gamma$ is a real $3$-dimensional cycle in $\C\P^3\setminus D$ for integrating over; e.g. the compact $3$-torus used in Section~\ref{s:IntroLus} below.  
Then our result implies that the $H^*(X,\C)$-valued function in $q$,
\[
  S(q) =   
  \left(\int_\Gamma e^{{\Wcanq}} p_3\,\omega \right)
 \ \sigma_{0} 
 +
  \left(\int_\Gamma e^{{\Wcanq}} p_2\,\omega \right) 
 \ \sigma_{1}+  
 \left(\int_\Gamma e^{{\Wcanq}} p_1\,\omega \right)
 \ \sigma_{2} +
  \left(\int_\Gamma e^{{\Wcanq}} \,\omega\right) 
  \ \sigma_{3},
\] 
satisfies the `flat section' differential equation
\[
 q\frac{d}{d q} S = \sigma_1\star_q  S.
\]
Here we have set $\p_0=1$ and used $\p_1,p_2, \p_3$ as coordinates on $\C\P^3\setminus D$, as in Section~\ref{s:QCoh1}.

We can improve on Theorem~\ref{t:connections} in the special case of $X=Q_3$ using a paper of Gorbounov and Smirnov. Namely in \cite{GS} Gorbounov and Smirnov construct their own ad hoc partial compactification of the original Givental-mirror \cite{Givental:EqGW} for $Q_{2m-1}$. They prove with Sabbah and Nemethi that their superpotential $\WGS$ is cohomogically tame, which implies that the associated Gauss-Manin system reconstructs the Dubrovin connection without needing to pass to a submodule, see \cite{sabbah}. We compare the  Gorbounov-Smirnov partial compactification of Givental's mirror to the canonical LG model 
and show that in the case of $Q_3$ they are isomorphic. 
Together with the result of \cite{GS}, we obtain in this case an isomorphism of the Gauss-Manin system of $(\Xcan,\Wcan)$ with the Dubrovin connection, see Theorem~\ref{t:connections}. 

For dimension greater than three the canonical LG model is only birationally isomorphic to the Gorbounov-Smirnov mirror. However it still has the expected number of critical points. We therefore conjecture that the canonical and the Lie-theoretic superpotentials are also cohomologically tame, and the isomorphism statement of Theorem~\ref{t:connections} extends to all quadrics.

\subsection{The Lie-theoretic superpotential and Lusztig coordinates}\label{s:IntroLus}
The flat sections $S(q)$ discussed above can also be written in terms of $(\Xlie,\Wlie)$, the Lie-theoretic LG model which was defined in \cite{rietsch}.  In this case the top degree coefficient of $S(q)$ (with respect to the grading on cohomology) takes on the form
\begin{equation}\label{e:integralLie}
  \left<S(q),\sigma_0\right> = \int_{\Gamma_{\mathrm {Lie}}\subset \Xlie} e^{\Wlieq}\omega.
\end{equation} 
Here $\Xlie$ is a $(2m-1)$-dimensional affine subvariety of the full flag variety $Sp_{2m}/B$ (it is a Schubert variety intersected with an opposite big cell), and $\omega$ is a particular holomorphic volume form on $\Xlie$. This formula \eqref{e:integralLie} which we prove here was conjectured in \cite[Conjecture~8.1]{rietsch}. 

Crucial to our proofs is a new Laurent-polynomial LG-model. 
%Like $\Xlie$, also $\Xcan$ is an intersections of opposite Bruhat cells in a homogeneous space for $\Sp_{2m}(\C)$. N
From the point of view of Lie theory, $\Xcan$ can be described as the image inside $\Xcheckcan=\C \P^{2m-1}$ of the intersection of two opposite Bruhat cells. It is therefore natural to restrict $(\Xcan,\Wcan)$ to a `Lusztig torus' $\Xlus$ inside $\Xcan$. That is we consider the same torus which would be used by Lusztig to parametrise the totally positive part of $\Xcheckcan$ viewed as $\Sp_{2m}$-homogeneous space, in the theory of total positivity~\cite{lusztig:TP}. After restriction of $\Wcan$ to this torus  we obtain a very nice Laurent polynomial, which is reminiscent of the standard superpotential for projective space $\C\P^n$,
\[z_1+\dotsc + z_n + q \frac 1{z_1\dotsc z_n},\] 
and which can be used to compute the integral \eqref{e:integralLie}. 

For example in the case of $X=Q_3$ we have a $3$-dimensional torus with coordinates $a,b,c$, and
\begin{equation}\label{e:LusCoordQ3}
\Wlusq=a+b+c +q \frac{a+b}{abc}.
\end{equation}
In the case of $X=Q_5$ we have a $5$-dimensional torus with coordinates $a_1,a_2,b_1,b_2,c$, and
\begin{equation}\label{e:LusCoordQ5}
\Wlusq=a_1+a_2+b_1+b_2+c +q \frac{a_1+b_1}{a_1a_2b_1b_2c}.
\end{equation}
We may rewrite our integral formulas in terms of these coordinates. For example if $X=Q_3$ we obtain the easily computable integral
\begin{equation*}
  \left<S(q),\sigma_0\right> =\oint \oint \oint e^{\left(a+b+c +q \frac{a+b}{abc}\right)}\frac{da}a\wedge\frac{db}b\wedge\frac{dc}c.
\end{equation*}

Finally, we note that analogous results in the parallel case of even quadrics are worked out  in~\cite{PRW:quadrics}. For even-dimensional quadrics, the Langlands dual homogeneous space is another even-dimensional quadric, thus the canonical mirror $(\Xcan,\Wcan)$ looks quite different from the one in the odd quadrics case. However when we restrict to the Lusztig torus in that setting, the formula is a straightforward generalisation of \eqref{e:LusCoordQ3}, \eqref{e:LusCoordQ5}, etc., to an even number of coordinates. 

\vskip .3cm

\noindent{\sl{Acknowledgements~:}}
The second author thanks George Lusztig for his great  PhD supervision and for introducing her to the theory of total positivity, which turns out to have so many beautiful connections. The second author also thanks Dale Peterson for his inspiring lectures on quantum cohomology.

\section{Overview of earlier LG models} \label{s:LGoverview}

We begin by recalling various earlier constructions of mirror Landau-Ginzburg models which are relevant in our setting.

\vskip .2cm

\paragraph{\bf The Givental mirror.} The earliest  Landau-Ginzburg model construction which applies to odd quadrics is due to Givental~\cite{Givental:EqGW}, who wrote down an LG  model  for any complex projective hypersurface $Y \hookrightarrow \C\P^N$. Givental's LG model is a regular function $\Wgiv$ on a hypersurface %$\Xgiv$
inside an $N$-dimensional torus.
%, which plays the part of the $B$-model to $Y$ in mirror symmetry, meaning the singularities of $(\Xgiv,\Wgiv)$ encode various structures to do with the Gromov-Witten theory of $Y$. For instance Givental's work shows that $(\Xgiv,\Wgiv)$ recovers the (small) $J$-function and quantum differential equations of $Y$. 

The odd quadric $Q_{2m-1}$ is a hypersurface inside $\C\P^{2m}$, and Givental's mirror is a regular function on a particular hypersurface in a $2m$-dimensional torus, namely 
\[
 \{ (x_1,...,x_{2m}) \mid x_{2m}+ \frac{q}{\prod_{i=1}^{2m} x_i} =1\}.
\]
In a more symmetric formulation the Givental mirror of $\Q_{2m-1}$ is 
\begin{align*}
	\Xgiv &= \left\{ (\nu_1,\dots,\nu_{2m+1}) \in (\C^*)^{2m+1} \mid \prod_{i=1}^{2m+1} \nu_i = q, \nu_{2m}+\nu_{2m+1} = 1 \right\},\\
	 \Wgiv &= \nu_1 + \dots + \nu_{2m-1}.
\end{align*}
Additionally, $\Xgiv$ comes with a holomorphic volume form. But we do not include it here as it will not be used later. 

\vskip .2cm

\paragraph{\bf The Przyjalkowski mirror}
We use the notation $(\Xprz,\Wprz)$ for a Laurent polynomial mirror written down in \cite{przyjalkowski} which extends Givental's mirror from $\Xgiv$ to a $(2m-1)$-dimensional torus containing it. In the case where $Y$ is the smooth quadric $\Q_3$ in $\mathbb P^4$ the Przyjalkowski mirror is given by
\[
	\Xprz := (\C^*)^3,\quad \Wprzq := Y_1+Y_2+\frac{(Y_3+q)^2}{Y_1 Y_2 Y_3}. 
\]
More generally for a quadric $\Q_{2m-1}$ the formula reads
\[
	\Xprz = (\C^*)^{2m-1},\quad \Wprzq = Y_1+\dots+Y_{2m-2}+\frac{(Y_{2m-1}+q)^2}{Y_1\dots Y_{2m-1}}.
\]
$\Wprz$  is obtained from $\Wgiv$ via the change of coordinates described in Section~\ref{s:Comparison}.

\vskip .4cm

One issue with both  $(\Xgiv,\Wgiv)$ and $(\Xprz,\Wprz)$ is that the superpotential df
does not in general have the expected number of critical points (at fixed generic value of $q$). Namely the expected number of critical points should be equal to $\dim H^*(\Q_{2m-1},\C)=2m$. The analogous problem in case of the even quadric $Q_4$ was already observed in \cite{EHX}. In this example \cite{EHX} constructed a partial compactification to solve this problem, albeit in an ad hoc fashion. 

The following LG models, described in detail later in the paper, are partial compactifications of Givental's mirror~$(\Xgiv,\Wgiv)$ which are known to have the correct number of critical points. 

\vskip .2cm
\paragraph{\bf The Lie-theoretic mirror.} The quadratic hypersurfaces $\Q_{2m-1}$ have a large symmetry group. Indeed $\Q_{2m-1}$ can be viewed as the Grassmannian of isotropic lines in $\C^{2m}$ for a fixed non-degenerate quadratic form. In this way the quadric $Q_{2m-1}$ is identified as a cominuscule homogeneous space for the group $G^\vee=\SO_{2m+1}$.
% or its universal cover $\Spin_{2m+1}(\C)$. 

For any projective homogeneous space of a complex algebraic group there is a Landau-Ginzburg model which was defined by the second author using a Lie-theoretic construction. Namely, in \cite{rietsch} a conjectural LG~model $(\Xlie,\Wlie)$ is constructed for any projective homogeneous space $X=G^\vee/P^\vee$ of a simple complex algebraic group $G^\vee$, as a regular function on an affine subvariety of the Langlands dual group $G$. This affine variety is generally larger than a torus. It is shown in \cite{rietsch} that  this Lie-theoretic LG model has the correct number of critical points. Namely its Jacobi ring is shown to recover the Peterson variety presentation \cite{peterson} of the quantum cohomology of $X$.

\vskip .2cm

\paragraph{\bf The Gorbounov-Smirnov mirror.} For odd-dimensional quadrics $\Q_{2m-1}$ a recent paper \cite{GS} of Gorbounov and Smirnov directly constructs a partial compactification $(\XGS,\WGS)$ of the Givental mirror, without making use of \cite{rietsch}. Moreover a version of mirror symmetry is proved, which identifies the initial data of the Frobenius manifold associated to their LG~model with that constructed out of the quantum cohomology. In particular the Gauss-Manin connection associated to $(\XGS,\WGS)$  is shown to be isomorphic to the small Dubrovin connection.

\section{Plan of the paper}\label{s:plan}

We begin in Section~\ref{s:Lie} by setting up notation and giving a careful definition of the Lie-theoretic superpotential $(\Xlie, \Wlie)$. The domain is an open subvariety of a $(2m-1)$-dimensional Schubert variety $\XcheckLie$ of the full flag variety of $\Sp_{2m}(\C)$.

\vskip .2cm
Our first result is that $\Wlie$ restricted to a certain torus recovers the Laurent polynomial superpotential $\Wlus$. This is proved in Section~\ref{s:Laurent}.

\vskip .2cm
We introduce the canonical LG model in Section~\ref{s:canIntro}. It has domain $\Xcan$ equal to the complement of an anti-canonical divisor in $\Xcheckcan=\C \P^{2m}$, where $\C \P^{2m}$ viewed as right homogeneous space for $\Sp_{2m}(\C)$. We then describe a birational map (depending on $q$) 
\begin{equation}\label{e:biratiso}
\XcheckLie - - - > \Xcheckcan.
\end{equation}
This birational map sends the torus used above isomorphically to the Lusztig torus $\Xlus$ in $\Xcheckcan$. We express the Lusztig coordinates in terms of the homogeneous coordinates of $\Xcheckcan$ and show that the formula for $\Wlus$ transforms to the formula for the canonical superpotential $\Wcan$. Therefore we see that $(\Xlie, \Wlie)$ and $(\Xcan, \Wcan)$ are birationally isomorphic (to each other as well as to $(\Xlus, \Wlus)$).

\vskip .2cm
Next, in Section~\ref{s:mirror} we show that the birational map \eqref{e:biratiso} restricts to an isomorphism
\[
\Xlie \to \Xcan.
\]
It follows that the canonical and the Lie-theoretic LG~models are isomorphic.

\vskip .2cm

Then in Section~\ref{s:Jacobi} we deduce an isomorphism which identifies the Jacobi ring associated to $(\Xcan,\Wcan)$ with the quantum cohomology ring of $X$. We show that under this isomorphism the homogeneous coordinates $\{\p_i\}$ map to the Schubert classes $\{\sigma_i\}$, and the quadratic denominators of $\Wcan$ each map to either~$q$ or to $\sigma_{2m-1}$.

\vskip .2cm
In Section~\ref{s:Comparison} we show that the following  LG models are all birationally equivalent:
\begin{itemize}
\item the Givental mirror $(\Xgiv,\Wgiv)$ from \cite{Givental:EqGW},
\item the Przyjalkowski mirror $(\Xprz,\Wprz)$ from \cite{przyjalkowski},
\item the Gorbounov-Smirnov mirror $(\XGS,\WGS)$ from \cite{GS},
\item and the canonical mirror $(\Xcan,\Wcan)$, or equivalently $(\Xlie, \Wlie)$ from \cite{rietsch}.
\end{itemize} 
In the case of $X=Q_3$ we also show that the Gorbounov-Smirnov mirror and the canonical mirror are isomorphic. 

\vskip .2cm
Section~\ref{s:consequences} is devoted to the Gauss-Manin system of $(\Xcan,\Wcan)$. Inside this Gauss-Manin system we identify a free $\C[q^{\pm 1}]$-submodule with connection, which is then shown to be isomorphic to the Dubrovin connection on $H^*(X,\C[q^{\pm 1}])$. From this result we deduce integral formulas for flat sections of the Dubrovin connection.
  
\vskip .2cm

Using results from \cite{GS} we deduce that the Dubrovin connection on $H^*(X,\C[q^{\pm 1}])$ is isomorphic to the Gauss-Manin system of $(\Xcan,\Wcan)$  (and not just to a submodule) in the special case of the quadric $Q_3$. In the final section we collect together and write out explicitly the formulas in the example of $Q_3$.
  
\section{The Lie-theoretic LG model  $(\Xlie,\Wlie)$}\label{s:Lie}

To introduce the  Lie-theoretic Landau-Ginzburg model 
we view the odd-dimensional quadric $X=\Q_{2m-1}$ for $m \geq 2$ as a homogeneous space under the special orthogonal group $G^\vee=\SO_{2m+1}(\C)$.
We fix a Borel subgroup $B^\vee_+$, a maximal torus $T^\vee$ and an opposite Borel subgroup $B^\vee_-$, and consider the Dynkin diagram of type~$B_m$:
\begin{center}
	\begin{tikzpicture}          
		\node[draw,fill=black,circle,minimum size=4pt,inner sep=0pt,label distance=2pt,label=below:{1}] (1) at (0,0) {};
		\node[draw,circle,fill=white,,minimum size=4pt,
                            inner sep=0pt,label distance=2pt,label=below:{2}] (2) at (1,0) {};
		\node[draw,circle,fill=white,,minimum size=4pt,
                            inner sep=0pt,label distance=2pt,label=below:{3}] (3) at (2,0) {};
		\node[draw,circle,fill=white,,minimum size=4pt,
                            inner sep=0pt,label distance=2pt,label=below:{$m-1$}] (4) at (3,0) {};
		\node[draw,circle,fill=white,minimum size=4pt,
                            inner sep=0pt,label distance=2pt,label=below:{$m$}] (m) at (4,0) {};
	
		\draw
			(1) -- (2)
			(2) -- (3);
	
		\draw[double distance=2pt]
			(4) -- (m);
	
		\draw[->,double distance=2pt,very thin]
			(4) -- (3.65,0);
	
		\draw[loosely dotted]
			(3) -- (4);
	\end{tikzpicture}
\end{center}
 We denote by  $P^\vee_{\omega_i} \supset B_+^\vee$   the  parabolic subgroup  corresponding to the $i$-th vertex of the diagram. The quadric $X=\Q_{2m-1}$ identifies with the homogeneous space $\SO_{2m+1}(\C)/P_{\omega_1}^\vee$.

The Landau-Ginzburg model for $X=\SO_{2m+1}(\C)/P^\vee_{\omega_1}$ defined in \cite{rietsch}, which we call the \emph{Lie-theoretic LG model} $(\Xlie,\Wlie)$, takes place on an affine subvariety $\Xlie$ of the Langlands dual flag variety. Let $G=\Sp_{2m}(\C)$ be the Langlands dual group of $G^\vee$, and $B_+$, $T$ and $B_-$ be the duals of $B^\vee_+$, $T^\vee$ and $B^\vee_-$, respectively. The Langlands dual flag variety is $\Sp_{2m}(\C)/B_-$, and the Lie-theoretic mirror $\Xlie$ is the intersection of two particular opposite open Bruhat cells in $\Sp_{2m}(\C)/B_-$. This intersection of cells is also called an \emph{open Richardson variety}. The Lie-theoretic potential $\Wlie$ will be a particular regular function on $\Xlie \times \C_q^*$, where by $\C^*_q$  we mean $\C^*$ with  coordinate denoted $q$.

\subsection{Notation for the symplectic group $G=\Sp_{2m}(\C)$.}\label{SpNotation} We denote by $\omega_i$ for $1 \leq i \leq m$ the $i$-fundamental weight of $\Sp_{2m}(\C)$, and by  $V_{\omega_i}$ the fundamental representation with highest weight $\omega_i$. We fix a basis $(v_1,\dots,v_{2m})$ for the representation $V:=V_{\omega_1} \cong \C^{2m}$ of $G$ with highest weight $\omega_1$ in such a way that the matrix of the symplectic form in the basis $(v_1,\dots,v_{2m})$ be given by
\[
 	J=\begin{pmatrix}
   		 &  &  &      &-1\\
   		 &   &  &1 & \\
  		 &&\Ddots & \\
   		 & -1 &  &  &\\
	    1& & &&
  	  \end{pmatrix}.
\] 
Then the Borel subgroups $B_+$ and $B_-$ consist of upper-triangular and lower-triangular matrices, and the maximal torus $T$ of diagonal matrices $(d_{ij})$ with non-zero entries $d_{i,i}=d_{2m-i+1,2m-i+1}\inv$. We also fix Chevalley generators $(e_i)_{1 \leq i \leq m}$ and $(f_i)_{1 \leq i \leq m}$ for the Lie algebra $\mathfrak{g}$ of $G$. Explicitly, we embed $\mathfrak{sp}(V,J)$ into $\mathfrak{gl}(V)$ and set
\[
	e_i := E_{i,i+1}+E_{2m-i,2m-i+1}  \ \text{ for $i=1,\dotsc, m-1$, \ \text{ and }\ } e_m := E_{m,m+1},
\] 
where $E_{i,j}=(\delta_{i,k}\delta_{l,j})_{k,l}$ is the standard basis of $\mathfrak{gl}(V)$. We also set $f_i:=e_i^T$, the transpose matrix, for every $i=1,\dotsc, m$. 

Using the Chevalley generators we introduce one-parameter subgroups of $G$ by setting $x_i(a):=\exp(a e_i)$ and $y_i(a):=\exp(a f_i)$. We choose specific representatives for elements of the Weyl group $W$ of $G$ by associating to a simple reflection $s_i$ the element
\[
 	\dot s_i=y_i(1)x_i(-1)y_i(1) \in G.
\]
If $s_{i_1}\cdots s_{i_r}$ is a reduced expression for $w \in W$ we denote by $\dot{w}$ the element of $G$ given by
\[
	\dot w=\dot s_{i_1}\cdots \dot s_{i_r},
\]
and we define $\ell(w) := r$, the length of the Weyl group element $w$. As is customary we also denote by $w_0$ be the longest element in $W$.

The $1$-parameter subgroups given by the $x_i$ generate $U_+$, and those given by the $y_i$ generate $U_-$. We define the following additive characters on $U_+$ and $U_-$, respectively,
\begin{equation*}
\begin{array}{cc}
e_i^*:U_+\to \C, &  \quad e_i^*(x_j(m))=m \delta_{i,j},\\
f_i^*:U_+\to \C, & \quad f_i^*(y_j(m))=m \delta_{i,j}.
\end{array}
\end{equation*}

 Recall that we realised our quadric $Q_{2m-1}$ as the homogeneous space $G^\vee/P^\vee_{\omega_1}$ for $SO_{2m+1}(\C)$. We now consider the \emph{dual parabolic subgroup} $P=P_{\omega_1}$ of $G=\Sp_{2m}(\C)$ associated with the first fundamental weight. Explicitly it is the subgroup whose Lie algebra is generated by all of the Chevalley generators $e_i$ together with $f_2,\dotsc, f_{m}$, leaving out $f_1$. We let $W_P$ denote the subgroup of the Weyl group $W$ associated with $P_{\omega_1}$, namely $W_P = \langle s_2,\dotsc, s_m \rangle$. We write $w_P$ for the longest element in $W_P$, $W^P$ for the set of minimal length coset representatives for $W/W_P$, and $w^P \in W^P$ for the minimal length coset representative of $w_0$.

\subsection{Definition of the Lie-theoretic LG model.}\label{s:LieNotation} 
In this section we follow \cite{rietsch}, adapting the results there to our special case. We first introduce the domain, $\Xlie \subset G / B_-$, of the Lie-theoretic mirror, namely 
\[ 
 	\Xlie := (B_+\dot w_P B_-  \cap B_-\dot w_0 B_- )/ B_-.
\]
It closure in $G/B_-$ is the Schubert variety
\[
\XcheckLie:=\overline{B_+\dot w_P B_-/B_- }.
\]
To write down the superpotential $\Wlie: \Xlie\x\C^*_q\to\C$, we introduce a variety $Z \subset G$ which is a covering of $\Xlie\x\C^*_q$. 
Let $T^{W_P}$ be the $W_P$-fixed part of the maximal torus $T$. Since $P$ is a maximal parabolic, this is a one-dimensional torus, and we have that 
\[
\begin{aligned}
\alpha_1:T^{W_P}&\to& \C^*_q,\\
t\ \ & \mapsto &\alpha_1(t)
\end{aligned}
\]
is a double cover. We set
\begin{equation}\label{e:Z}
 	Z:=B_-\dot w_0\cap U_+ T^{W_P}\dot w_P U_-,
\end{equation}
and define a map
\begin{equation*}
\begin{array}{cccc}
\pilie: &Z\qquad &\to &\Xlie\x\C^*_q\ \\
&z=u_1 t\dot w_p \bar u_2\ &\mapsto &\quad (zB_-,\alpha_1(t)).
\end{array}
\end{equation*}
which is again a double cover. Note that if we were to quotient out $Z$ by the action of the centre, $\{\pm \mathbf 1\}$, of $\Sp_{2m}(\C)$, then the map would be an isomorphism.  This would be the convention taken in \cite{rietsch}.

We define a regular function on $Z$ by
\begin{equation*}
\begin{array}{cccc}
\mathcal F:& Z\qquad &\to &\C\ \\
&z=u_1 t\dot w_p \bar u_2\ &\mapsto &  \sum e_i^*(u_1)+\sum f_i^*(\bar u_2). 
\end{array}
\end{equation*}
From \cite{rietsch} it follows that $\mathcal F$ is well-defined and descends to a regular function $\Wlie:\Xlie\x\C^*_q\to \C$ such that the diagram 
\begin{equation*}
\begin{array}{rccc}
&\quad Z&&\\
&\pilie\downarrow &\overset{\mathcal F}\searrow &\\
\Wlie:&\Xlie\x\C^*_q&\longrightarrow &\C\ 
\end{array}
\end{equation*}
commutes.
The corresponding map  for fixed $q$, is denoted
\[
 	\Wlieq : \Xlie \to \C,\qquad u_1\dot w_P B_- \mapsto \Wlie(u_1\dot w_P B_-,q).
\]

\section{The Laurent polynomial superpotential $\Wlus$.}\label{s:Laurent}
We continue using all of the notations from the previous section. We want to restrict the Lie-theoretic superpotential $\Wlie$ to a well-chosen torus to obtain a particularly nice Laurent polynomial. Instead of constructing the torus inside $\Xlie$, we will use $Z$, the double cover of $\Xlie\x\C^*$.  Recall that $Z\subset B_-\dot w_0$ consists of those elements $z$ which can be written in the form
\[
z=u_1 t\dot w_P\bar u_2
\]
for $u_1\in U_+, t\in T^{W_P}$ and $\bar u_2\in U_-$. However, the factors $u_1$ and $\bar u_2$ in this factorisation are not uniquely determined. We can make them uniquely determined for example by restricting the domain of $\bar u_2$, which is what we will do now. 
Let 
\begin{equation}\label{e:UPminus}
	U_-^P:=U_-\cap B_+(\dot w^P)\inv B_+. 
\end{equation}
We have the following proposition.
\begin{prop}\label{p:u2barz}
For any $z\in Z$ there exists a unique $\bar u_2=\bar u_2(z)\in U_-^P$ such that $z$ has a factorisation of the form
\[
z=u_1 t\dot w_P  \bar u_2.
\]
We  also write $t(z)=t$ if $z$ is factored as above.  The map $\theta:Z \to U_-^P\x  T^{W_P}$ defined by 
\[
\theta: z\mapsto (\bar u_2(z),t(z))\] 
is an isomorphism of affine varieties.  
\end{prop}

This proposition is proved using the twist map of Berenstein and Zelevinsky. 

\begin{theo}[{\cite[Theorem 1.2]{BeZel:TotPos}}]\label{t:BZ}
	Let $G$ be a semisimple algebraic group. Let $B_+,B_-$  be  opposite Borel subgroups in $G$ and $U_+,U_-$ their unipotent radicals. Denote by $\dot w \in G$ a choice of representative for an element $w$ of the Weyl group $W$ as in Section~\ref{SpNotation}. Consider $y \in U_- \cap B_+ \dot w^{-1} B_+$. There exists a unique $x \in U_+ \cap B_- \dot w B_-$ such that $U_+ \cap B_- \dot w y=\{x\}$, and the resulting map
	\[
		\tilde \eta_w : U_- \cap B_+ \dot w^{-1} B_+ \to U_+ \cap B_- \dot w B_-, y \mapsto x
	\]
	is an isomorphism. In particular there exists an inverse isomorphism
	\[
		\epsilon_w : U_+ \cap B_- \dot w B_- \to U_- \cap B_+ \dot w^{-1} B_+.
	\]
\end{theo}

\begin{proof}[Proof of Proposition~\ref{p:u2barz}]
The map $z\mapsto (t(z),\bar u_2(z))$ is constructed as the composition
\[
	Z \to (U_- \dot w_0 \cap B_+ \dot w_P U_-)\x T^{W_P} \to (U_+ \cap B_- \dot w^P B_-)\x  T^{W_P}\to  U_-^P\x T^{W_P},
\]
where the rightmost map is defined using the isomorphism $\epsilon_{w^P}$ from Theorem~\ref{t:BZ}, the middle map is defined using left multiplication by $\dot w_0^{-1}$, which is also an isomorphism, and the leftmost map is defined to be 
\[
z=b_- \dot w_0 \mapsto ([b_-]^{-1}_0 b_- \dot w_0, t(z)).
\] 
Here $[b_-]_0$ denotes the torus part of the Borel group element $b_-\in T U_-$. This latter map is also an isomorphism with inverse $(b_+ \dot w_P u_-,t) \to t[b_+]_0^{-1}b_+ \dot w_P u_-$.%, where as usual $t \in T^{W_P}$ is such that $\alpha_1(t)=q$. We may now start the proof of the theorem.
\end{proof}

\subsection{The intermediate LG model} \label{s:Wtheta}
We now have the commutative diagram 
\begin{equation*}
\begin{array}{ccccccc}
\Xlie\x\C^*_q & \overset{\pilie}{\longleftarrow} &  Z &\overset{\theta}{
\underset{\cong}\longrightarrow}& U_-^P\x T^{W_P} & \overset{\operatorname{id} \x\alpha_1}\longrightarrow &  \quad U_-^P\x \C^*_q
 \\
%\operatorname{id} \x\alpha_1\uparrow &\searrow^{\Wtheta}&\\
% &\overset{\Ftheta}\longrightarrow\qquad & \C\\
%\quad\theta\uparrow\cong   & & \|\\
%&\overset{\mathcal F}\longrightarrow\qquad& \C\\
%\pilie\downarrow &\nearrow_{\Wlie}& \\
%&& \\
\Wlie\downarrow\qquad & &\mathcal F \downarrow\qquad & &\mathcal F_\theta \downarrow \qquad & & \Wtheta\downarrow\qquad \\
\C &=& \C &=& \C &=& \C,
\end{array}
\end{equation*}
defining the maps $\Ftheta: U_-^P\x T^{W_P}\to \C$, and $\Wtheta: U_-^P\x \C^*_q\to \C$. Note that we may invert $\pilie$ and $\operatorname{id}\x \alpha_1$, if we quotient out by the action of $\{\pm \mathbf 1\}$ on $Z$ and $T^{W_P}$. Therefore we may think of 
$(U_-^P, \Wtheta)$ as being an isomorphic LG model to $(\Xlie,\Wlie)$; there is an isomorphism $U_-^P\x\C^*_q\to \Xlie\x\C^*_q$ which is the identity on the second factor, and under which  $\Wlie$ pulls back to $\Wtheta$. 

\subsection{The Laurent polynomial LG model}\label{s:Wlus}
We define an open dense torus inside $U_-^P$ as follows. The Weyl group element $w^P \in W^P$ has the reduced expression
\[
	w^P = s_1 s_2 \dots s_{m-1} s_m s_{m-1} \dots s_2 s_1.
\]
As a consequence of this and the Bruhat lemma, a generic element $\bar u_2$ in $U_-^P$ can be written as a product of elements of $1$-parameter subgroups as follows,
\begin{equation}\label{e:u2barfactb}
	\bar u_2 = y_1(a_1) \dots y_{m-1}(a_{m-1}) y_m(c) y_{m-1}(b_{m-1}) \dots y_1(b_1),
\end{equation}
where $a_i, c, b_j \neq 0$. Thus we define the torus $\mathcal T \subset U_-^P$  to be
\[
	\mathcal T := \{  y_1(a_1) \dots y_{m-1}(a_{m-1}) y_m(c) y_{m-1}(b_{m-1}) \dots y_1(b_1) \mid a_i,c,b_i \in \C^* \}.
\]
Before working out the restriction of the superpotential to this torus, we note that it is natural to think of $\mathcal T$ as embedded in the homogeneous space 
\begin{equation}\label{e:Xcheckcan}
\Xcheckcan:= P\backslash Sp_{2m}(\C)\cong \C\P^{2m},
\end{equation}
via $ \bar u_2 \mapsto  P_{\omega_1} \bar u_2$, setting the stage for the {\it canonical superpotential} to be introduced in the next section. 

 Thus we make the following definition.
 \begin{defn}\label{d:XlusWlus}
We denote the image of the torus $\mathcal T$ in $\Xcheckcan$ by $\Xlus$, and denote the coordinates on $\Xlus$ in the same way as those on $\mathcal T$, by $a_i, c, b_j \neq 
0$. Explicitly,
\[
	\Xlus := \{ P y_1(a_1) \dots y_{m-1}(a_{m-1}) y_m(c) y_{m-1}(b_{m-1}) \dots y_1(b_1) \mid a_i,c,b_i \in \C^* \}.
\]
The restriction of $\Wtheta$ to $\mathcal T$ defines a map
\[
\Wlus:\Xlus\x\C^*_q\to\C.
\]
\end{defn}

\begin{theo}\label{t:LaurentLie}
	In terms of the coordinates $a_i,b_i, c$ on  $\Xlus$,  
  	\begin{equation}\label{e:W1}
     	\Wlus = a_1 + \dots + a_{m-1} + c + b_{m-1} + \dots + b_1 + q \frac{a_1+b_1}{a_1 \dots a_{m-1} c b_{m-1} \dots b_1}.
  	\end{equation} 
\end{theo}

\begin{proof}[Proof of Theorem~\ref{t:LaurentLie}]
	Consider an element $\bar u_2 \in\mathcal T\subset U_-^P$ and choose a $t\in T^{W_P}$ such that $\alpha_1(t)=q$. By definition, $\bar u_2$ admits a factorisation 
	\[
		\bar u_2 = y_1(a_1) \dots y_{m-1}(a_{m-1}) y_m(c) y_{m-1}(b_{m-1}) \dots y_1(b_1).
	\]
Let $z := \theta^{-1}(\bar u_2,t)$, where $\theta$ is the isomorphism from Proposition~\ref{p:u2barz}. Then $z$ can be written as $z=u_1 t \dot w_P \bar u_2$ for some unique $u_1 \in U_+$, and 
	\[
		\Wtheta(\bar u_2,q) = \mathcal{F}(u_1 t \dot w_P \bar u_2) = \sum_{i=1}^m e_i^*(u_1) + \sum_{i=1}^m f_i^*(\bar u_2).
	\]
	The theorem now follows from the lemma below.
	\end{proof}
	
	\begin{lemma}\label{p:eis}
 		If $u_1$ and $\bar u_2$ are as above then we have the following identities  
 		\begin{align}\label{eq:ei1first}
  			f_i^*(\bar u_2) = \begin{cases}
                     			a_i + b_i & \text{if $1 \leq i \leq m$,} \\
                     			c & \text{otherwise.}
                   			  \end{cases}
 		\end{align}
 		\begin{align}\label{eq:ei2first}
  			e_i^* (u_1) = \begin{cases}
                  			0 & \text{if $2 \leq i \leq m$,} \\
                  			q \frac{a_1 + b_1}{a_1 \dots a_{m-1} c b_{m-1} \dots b_1} & \text{if $i=1$}.
                 		  \end{cases}
 		\end{align}
 	\end{lemma}
 
	\begin{proof}%\let\qed\relax
 		Equation \eqref{eq:ei1first} is obtained immediately from the definition of $\bar u_2$. For Equation \eqref{eq:ei2first}, let $v_{\omega_{i}}^- $ and $ v_{\omega_i}^+$ denote a lowest, respectively highest, weight vector in $V_{\omega_i}$ and notice that
 		\begin{align*}
  			e_i^*(u_1) & = \frac{\langle u_1^{-1} \cdot v_{\omega_i}^- , e_i \cdot v_{\omega_i}^-\rangle}{\langle u_1^{-1} \cdot v_{\omega_i}^- , v_{\omega_i}^-\rangle} \\
	     			& = \frac{\langle t \dot w_P \bar u_2 \cdot v_{\omega_i}^+ , e_i \cdot v_{\omega_i}^-\rangle}{\langle t \dot w_P \bar u_2 \cdot v_{\omega_i}^+ , v_{\omega_i}^-\rangle}.
 		\end{align*}
		Assume $2 \leq i \leq m$. Then $e_i^*(u_1) = 0$ if and only if $\langle \bar u_2 \cdot v_{\omega_i}^+ , \dot w_P^{-1} e_i \cdot v_{\omega_i}^-\rangle = 0$. Now the vector $w_P^{-1} e_i \cdot v_{\omega_i}^-$ is in the $\mu$-weight space of the $i$-th fundamental representation, where $\mu = w_P^{-1} s_i(-\omega_i)$. Moreover, $\bar u_2 \in B_+ (\dot w^P)^{-1} B_+$, hence $\bar u_2 \cdot v_{\omega_i}^+$ can have non-zero components only down to the weight space of weight $(w^P)^{-1}(\omega_i) = w_P^{-1}(-\omega_i)$. Since $\ell(w_P^{-1} s_i) > \ell(w_P^{-1})$ for $2 \leq i \leq m$, this is higher than $\mu$, which proves that $e_i^*(u_1) = 0$.

 		Now assume $i=1$. We have
 		\begin{align*}
  			e_1^*(u_1) & = \frac{\langle t \dot w_P \bar u_2 \cdot v_{\omega_1}^+ , e_1 \cdot v_{\omega_1}^-\rangle}{\langle t \dot w_P \bar u_2 \cdot v_{\omega_1}^+ , v_{\omega_1}^-\rangle} \\
	     			& = (\omega_1 + \alpha_1 - \omega_1)(t) \frac{\langle \bar u_2 \cdot v_{\omega_1}^+ , \dot w_P^{-1} e_1 \cdot v_{\omega_1}^-\rangle}{\langle \bar u_2 \cdot v_{\omega_1}^+ , \dot w_P v_{\omega_1}^-\rangle} \\
             		& = q \frac{\langle \bar u_2 \cdot v_{\omega_1}^+ , \dot w_P^{-1} e_1 \cdot v_{\omega_1}^-\rangle}{\langle \bar u_2 \cdot v_{\omega_1}^+ , v_{\omega_1}^-\rangle}.
 		\end{align*}
 		First look at the denominator. The only way to go from the highest weight vector $v_{\omega_1}^+$ of the first fundamental representation to the lowest $v_{\omega_1}^-$ is to apply $g \in B_+ w B_+$ for $w \geq (w^P)^{-1}$. Since $\bar u_2 \in B_+ (\dot w^P)^{-1} B_+$, it follows that we need to use all factors of $\bar u_2$, and normalising $ v_{\omega_1}^-$ appropriately, we get
 		\[
  			\langle \bar u_2 \cdot v_{\omega_1}^+ , v_{\omega_1}^-\rangle = a_1 \dots a_{m-1} c b_{m-1} \dots b_1.
 		\]
 		Finally, we look at the numerator $\langle \bar u_2 \cdot v_{\omega_1}^+ , \dot w_P^{-1} e_1 \cdot v_{\omega_1}^-\rangle$. Let $\epsilon_i$ denote the weight of the basis vector $v_i\in V_{\omega_1 }$ when $1\le i\le m$.  The vector $\dot w_P^{-1} e_1 \cdot v_{\omega_1}^-$ has weight 
 		\[
  			\mu' = \dot w_P^{-1} s_1 (-\omega_1) = \dot w_P^{-1} (-\epsilon_2) = \epsilon_2.
 		\]
Indeed, $\dot w_P\inv e_1\cdot v^-_{\omega_1}=v_2$. From the definition of $\bar u_2$, it follows that $\langle \bar u_2 \cdot v_{\omega_1}^+ , v_2\rangle = a_1 + b_1$, which concludes the proof of the lemma and of the theorem.
		\end{proof}

In this section we have  re-expressed  the LG model  $(\Xlie, \Wlie)$ in terms of a regular function on a subvariety of $U_-$, namely we introduced the intermediate LG model $(U^P_-,\Wtheta)$. Then we restricted to a natural choice of torus inside $U^P_-$ to find a simple Laurent polynomial expression, leading us to $(\Xlus,\Wlus)$. 

We are now ready to introduce the canonical mirror.

\section{Construction of the canonical LG model $(\Xcan,\Wcan)$.} \label{s:canIntro}

We now construct the canonical LG model $(\Xcan,\Wcan)$ and state our main comparison theorem. 

\subsection{$\Xcheckcan$ and its affine subvariety $\Xcan$}\label{s:Xcan}
Recall the definition of $\Xcheckcan$ from~\eqref{e:Xcheckcan}, as right homogeneous space for $\Sp_{2m}(\C)$,
\[
\Xcheckcan=P_{\omega_1}\backslash \Sp_{2m}(\C).
\]
If $V=\C^{2m}$ is the defining representation of $\Sp_{2m}(\C)$ as in Section~\ref{SpNotation}, then  $\Xcheckcan$ may equivalently be described as $\P(V^*)$, viewed as an orbit of $\Sp_{2m}(\C)$ acting from the right. 

\begin{remark} We note that on $V^*$ we have both the action from the right (matrix multiplication from the right on the vector space of row vectors), and the action from the left (dual representation of $V$). Namely these are related by $g\cdot v^*=v^*\cdot g\inv$ for $v^*\in V^*$.
\end{remark}

We want to choose fixed coordinates on $\Xcheckcan$. Let $v^0=(0,\dotsc,0,1)$ in $V^*$, so that the line $\langle v^0\rangle_\C\in\P(V^*)$ has stabiliser $P=P_{\omega_1}$.   
We let $w_{(k)}\in W$ be defined by
 	\begin{align*}
 		w_{(k)} = \begin{cases}
        			s_1 s_{2} \dots s_k & \text{if $k \leq m$,} \\
         		s_1 s_2  \dots\dots s_{m-1} s_m s_{m-1}\dotsc s_{2m-k}  & \text{if $m+1 \leq k \leq 2m-1$.}
        		  \end{cases}
 	\end{align*}
This defines a total ordering on the minimal length coset representatives for $W_P\backslash W$. It gives rise to a basis $\{v^0,\dotsc, v^{2m-1}\}$ of $V^*$ where $v^k:= v^0 \cdot \dot w_{(k)}$. Explicitly,
\begin{equation}\label{e:basisVstar}
v^0=(0,\dotsc, 0,1), v^1=(0,\dotsc, 1,0), \dotsc, v^{2m-1}=(1,0,\dotsc, 0).
\end{equation}
We can now introduce notation for the homogeneous coordinates of an element  $x\in \Xcheckcan $, described as a coset $x = P g$, by using the identification $\Xcheckcan=\P(V^*)$ and the basis~\eqref{e:basisVstar}.	
 \begin{defn}\label{def:plucker}
 	For $0 \leq k \leq 2m-1$ and $g\in \Sp_{2m}(\C)$ define the homogeneous coordinates %`Pl\"ucker coordinate' 
 	$\p_k(g)$ for the coset $P g \in \Xcheckcan$ by
 	\[
 		\p_k(g) = \langle v^0 \cdot g, v^k \rangle.
 	\]
Here the angle brackets refer to the coefficient with respect to the basis $\{v^0,\dotsc, v^{2m-1}\}$. 
\end{defn}
Applying this definition, the  homogeneous  coordinates for a coset $P g \in \Xcheckcan$ are just given by the bottom row entries of the matrix $g$ read from right to left. We note that, if as before we write $g$ as $g=u_1t\dot w_P \bar u_2$, then 
\[
	(\p_0(g):\dotsc : \p_{2m-1}(g)) =(\p_0( \bar u_2):\dotsc:\p_{2m-1}( \bar u_2)),
\]
since $Pg=P\bar u_2$. Changing the coset representative only rescales all of the homogeneous coordinates by a common factor.

Finally, note that since the basis elements are of the form $v^k=v^0 \cdot \dot w_{(k)}$ and  these homogeneous coordinates can also be interpreted as generalised minors.

\begin{defn}\label{d:Xcan} We define an affine subvariety of $\Xcheckcan$ by 
\begin{equation*}
 	\Xcan:= \Xcheckcan \setminus D,
\end{equation*}
where $D:=D_0\cup D_1\cup \dotsc \cup D_{m-1}\cup D_m$, the divisors $D_i$ being given by
\begin{align*}
	& D_0 := \left\{ \p_0 = 0 \right\}, \\
	& D_\ell := \left\{ \p_\ell \p_{2m-1-\ell} - \p_{\ell-1} \p_{2m-\ell} + \dots + (-1)^\ell \p_0 \p_{2m-1} = 0 \right\} \text{ for $1\leq \ell \leq m-1$,}\\
 	& D_m := \left\{ \p_{2m-1} = 0 \right\}.
\end{align*}
The divisor $D$ is an anticanonical divisor. Indeed, the index of $\X=\C\P^{2m-1}$ is $2m$. We may also use the notation $\delta_\ell$ for the quadratic expression 
\[
	\delta_\ell := \p_\ell \p_{2m-1-\ell} - \p_{\ell-1} \p_{2m-\ell} + \dots + (-1)^\ell \p_0 \p_{2m-1},
\]
where $0 \leq \ell \leq m-1$. 
\end{defn}

\subsection{The superpotential $\Wcan$ and the isomorphism theorem}\label{s:WcanDef}

In the previous section we defined $\Xcan$ and $\Xcheckcan$ and the homogeneous coordinates $p_i$. To define the `canonical' LG model $(\Xcan, \Wcan)$ it remains to define the superpotential $\Wcan$. 
\begin{defn}
$\Wcan$ is defined to be the  regular map $\Wcan : \Xcan \x \C^*_q\to \C$ expressed in terms of the homogeneous coordinates of $\Xcheckcan$ by 
    \begin{equation}\label{eq:W2}
      	\Wcan = \frac{\p_{1}}{\p_{0}} + \sum_{\ell=1}^{m-1} \frac{\p_{\ell+1} \p_{2m-1-\ell}}{\p_\ell \p_{2m-1-\ell} - \p_{\ell-1} \p_{2m-\ell} + \dots + (-1)^\ell \p_0 \p_{2m-1}} + q\frac{\p_{1}}{\p_{2m-1}}.
   	\end{equation}
If $q$ is fixed we use the notation $\Wcanq$. We refer to the pair $(\Xcan,\Wcan)$ as the `canonical' LG~model.
\end{defn}

The next two sections will be devoted to proving the following comparison theorem between $(\Xlie,\Wlie)$ and $(\Xcan,\Wcan)$. Recall the definition of the subvarieties $Z\subset B_-\dot w_0$ and $U^P_-\subset U_-$ from \eqref{e:Z} and \eqref{e:UPminus}, respectively. 

By Proposition~\ref{p:u2barz} and Section~\ref{s:Wtheta} we have an isomorphism  	$\Xlie\x \C^*_q\overset{\sim}\longrightarrow U_P^-\x\C^*_q$. We use it to define a `comparison' map
 	\begin{equation}\label{e:Psi}
\Psi: 	\Xlie\x \C^*_q\overset{\sim}\longrightarrow U_P^-\x\C^*_q\overset{\pi_{\operatorname{can}}}\longrightarrow  \Xcheckcan\x\C_q^*
 	 	\end{equation}
 	 	where the right hand side map is simply defined by $(\bar u_2,q)\mapsto (P\bar u_2,q)$. All in all the map $\Psi$ is given by
 	 	\begin{equation*}
 	 	(gB_-,q)=(u_1t\dot w_P \bar u_2B_-,q)\mapsto (\bar u_2,q)\mapsto (P\bar u_2,q),
 	 	\end{equation*}
 	 	where $z=u_1t\dot w_P \bar u_2$ is the uniquely (up to $\pm\mathbf 1$) determined element of $Z$ for which $\alpha_1(t)=q$ and $zB_-=gB_-$.
\begin{theo}\label{t:canLie}
 	The map $\Psi$ from \eqref{e:Psi} has image $\Xcan\x \C^*_q$ and defines an isomorphism between $\Xlie\x \C_q^*$ and $\Xcan\x\C_q^*$ such that the following diagram commutes
 	\[
\begin{array}{ccc}
\Xlie\x \C_q^*&\overset{\Psi}\longrightarrow &\Xcan\x\C_q^*\\ 
\qquad\downarrow\Wlie & & \qquad\downarrow \Wcan\\
\C&= & \C.
\end{array} 	
 	\]
In other words, we have an isomorphism $\Xlie\x\C^*_q\to \Xcan\x\C^*_q$ which is the identity on the second factor, and under which  $\Wcan$ pulls back to $\Wlie$.  	\end{theo}

\section{The superpotential $\Wcan$ and the Laurent polynomial $\Wlus$}\label{s:mirror}

In this section we prove a birational version of Theorem~\ref{t:canLie}. To do this we make use of the Laurent polynomial LG model $(\Xlus,\Wlus)$ which is birational to $(\Xlie,\Wlie)$ by construction.

\begin{prop}\label{p:Wrat}
Consider the canonical superpotential $\Wcan$ as the rational function on $\Xcheckcan\x\C^*_q$ given by,
 	\[
 		\Wcanq = \frac{\p_{1}}{\p_0} + \sum_{\ell=1}^{m-1} \frac{\p_{\ell+1} \p_{2m-1-\ell}}{\p_\ell \p_{2m-1-\ell} - \p_{\ell-1} \p_{2m-\ell} + \dots + (-1)^\ell \p_0 \p_{2m-1}} + q\frac{\p_{1}}{\p_{2m-1}}.
 	\]
The restriction of $\Wcan$ to the Lusztig torus $\Xlus\subset\Xcheckcan$ is regular and agrees with $\Wlus$, see Definition~\ref{d:XlusWlus} and Theorem~\ref{t:LaurentLie}. 
\end{prop}
Note that the Laurent polynomial superpotential $\Wlus$ was in fact obtained from $\Wlie$ by restriction of $\Wlie$ to the torus $\Psi\inv(\Xlus\x\C^*_q)$, compare Section~\ref{s:Laurent}. Thus Proposition~\ref{p:Wrat} has the following Corollary.

\begin{cor}\label{c:Wrat} We have a commutative diagram of (rational) maps
 	\[
\begin{array}{ccccc}
\Psi:\Xlie\x \C_q^*&\longrightarrow &\Xlus\x\C_q^* &\hookrightarrow& \Xcheckcan\x\C_q^*\\ 
\qquad\downarrow\Wlie & &\qquad\downarrow \Wlus& & \qquad\downarrow \Wcan\\
\C&= & \C& = &\C.
\end{array} 	
 	\]

\end{cor} \qed

\begin{proof}[Proof of Proposition~\ref{p:Wrat}]
	Consider an element $x=P\bar u_2 \in \Xlus$. By definition of $\Xlus$, $\bar u_2$ admits a factorisation 
	\[
		\bar u_2 = y_1(a_1) \dots y_{m-1}(a_{m-1}) y_m(c) y_{m-1}(b_{m-1}) \dots y_1(b_1).
	\]
	Recall from Definition~\ref{def:plucker} that the homogeneous coordinates are given by
	\[
		\p_k(x) = \langle v^0 \cdot \bar u_2, v^k \rangle=\langle v^0 \cdot \bar u_2,v^0 \cdot \dot w_{(k)}\rangle.
	\]
	Using the factorisation of $\bar u_2$ the following result is immediate.
	\begin{lemma}\label{l:bislemma} 
		Let $0 \leq k \leq 2m-1$ be an integer. Then if $\bar u_2$ possesses a factorisation of the form \eqref{e:u2barfactb} we have
 		\begin{align*}
  			\p_k(\bar u_2) = \begin{cases}
		    						1 & \text{if $k=0$,} \\
                    				a_1 \dots a_{k-1} (a_k + b_k) & \text{if $1 \leq k \leq m-1$,} \\
                    				a_1 \dots a_{m-1} c b_{m-1} \dots b_{2m-k} & \text{otherwise.} 
                   			\end{cases} 
 		\end{align*}\hskip 5cm \vskip -1cm\qed
	\end{lemma}
	\vskip .2cm
	By an easy computation it follows that 
	\[
		\Wcanq(x)=\frac{\p_{1}}{\p_0} + \sum_{\ell=1}^{m-1} \frac{\p_{\ell+1} \p_{2m-1-\ell}}{\p_\ell \p_{2m-1-\ell} - \p_{\ell-1} \p_{2m-\ell} + \dots + (-1)^\ell \p_0 \p_{2m-1}} + q\frac{\p_{1}}{\p_{2m-1}}=\Wlusq(x).
	\]
	Since $\Xlus$ is open dense in $\Xcheckcan$ this completes the proof of the proposition.
\end{proof}

In the next section we will study the locus $\Xcan$ where $\Wcanq$ is regular and prove that $\Psi$ is an isomorphism $\Xlie\x\C^*_q\to\Xcan\x\C^*_q$. This will complete the proof of Theorem~\ref{t:canLie}.

\section{The canonical mirror variety}
\newcommand{\omegacan}{\omega_{can}}

Recall that a element in $\Xcheckcan=P\backslash SL_{2m}(\C)$ has projective coordinates $(\p_0:\p_1:\dots:\p_{2m-1})$ which were introduced in Definition~\ref{def:plucker}.  Then the affine subvariety $\Xcan$ was defined as the complement of a particular anticanonical divisor $D \subset \Xcheckcan$ expressed in these coordinates. Namely  
\[	D:=D_0+D_1+\dotsc + D_{m-1}+ D_m,
\]
with $D_0=\{p_0=0\}, D_m=\{\p_{2m-1}=0\}$ and each remaining $D_\ell$ of the form $D_{\ell}=\{\delta_\ell=0\}$ for a particular quadratic polynomial $\delta_{\ell}$, see Definition~\ref{d:Xcan}. We let $\delta_0=\p_0\p_{2m-1}$.
%where the irreducible components $D_i$ are given by the equations
%\begin{align*}
%	D_0 &:= \left\{ \p_0 = 0 \right\}, \\
%	D_\ell &:= \left\{ \p_\ell \p_{2m-1-\ell} - \p_{\ell-1} \p_{2m-\ell} + \dots + %(-1)^\ell \p_0 \p_{2m-1} = 0 \right\} \text{for $1\leq \ell \leq m-1$},\\
% 	D_m &:= \left\{ \p_{2m-1} = 0 \right\}.
%\end{align*}

The goal of this section is to prove the following proposition.

\begin{prop}\label{p:iso}
 	The map from \eqref{e:Psi} defines an isomorphism $\Psi: \Xlie\x\C^*_q \to \Xcan\x \C^*_q$.
\end{prop}

\begin{proof} 
By construction, the canonical superpotential $\Wcanq$ is regular on $\Xcan$ and this is the whole regular locus.
Now from Corollary~\ref{c:Wrat} it follows that $\Psi$ maps $\Xlie\x\C^*_q$ into the regular locus of $\Wcan$. Therefore $\Psi$ must have its image in $\Xcan\x\C^*_q$. We will now prove the result by constructing an inverse map $\Psi\inv:\Xcan\x\C^*_q\to \Xlie\x\C^*_q$.

It suffices to set $q=1$ and construct a map $\Psi_{q=1}\inv:\Xcan\to\Xlie$ for which 
\[
\Psi_{q=1}\inv (P\bar u_2)= zB_-
\]
where $z=u_1\dot w_P\bar u_2\in Z$.  Then $\Psi\inv$ is constructed out of $\Psi_{q=1}\inv$ by setting 
\[
\Psi\inv:\ (P\bar u_2,q)\ \mapsto \ t(q)\Psi_{q=1}\inv (P\bar u_2)\ =\  t(q)zB_-,
\]
where for any $q$ we let $t(q)\in T^{W_P}$ be a torus element for which $\alpha_1(t(q))=q$.

 To construct $\Psi_{q=1}\inv$ we consider the morphism
	\[
		\Phi : \Xcan \to B_- \dot w_0,\qquad (\p_0 : \p_1 : \dots : \p_{2m-1}) \to \Phi(\p_0 : \p_1 : \dots : \p_{2m-1}),
	\]
	where $\Phi(\p_0 : \p_1 : \dots : \p_{2m-1})$ is the matrix of the linear map which to the basis element $v_j$ of $V \cong \C^{2m}$ associates
	\begin{align*}
		\begin{cases}
			\p_{2m-1} v_{2m} & \text{ if $j=1$,} \\
 			(-1)^{j-1}\frac{\delta_{j-1}}{\delta_{j-2}} v_{2m+1-j} + \p_{2m-j}\left(\sum_{\ell=1}^{j-2} (-1)^{\ell}\frac{\p_\ell}{\delta_{\ell-1}}v_{2m-\ell} + v_{2m}\right) &\text{ if $2\leq j\leq m$,} \\
 			(-1)^{j-1}\frac{\delta_{2m-1-j}}{\delta_{2m-j}}v_{2m+1-j} \\
 			+ \p_{2m-j}\left( \sum_{\ell=m+1}^{j-1} (-1)^{\ell-1}\frac{\p_{\ell-1}}{\delta_{2m-\ell}}v_{2m+1-\ell} + \sum_{\ell=1}^{m-1}(-1)^{\ell}\frac{\p_\ell}{\delta_{\ell-1}}v_{2m-\ell} + v_{2m}\right) &\text{if $m+1\leq j\leq 2m-1$,} \\
 			\frac{-1}{\delta_0} v_1 + \sum_{\ell=1}^{m-1} (-1)^{\ell+1}\frac{\p_{2m-1-\ell}}{\delta_\ell} v_{\ell+1} + \sum_{\ell=1}^{m-1} (-1)^{\ell} \frac{\p_\ell}{\delta_{\ell-1}} v_{2m-\ell} + v_{2m} &\text{ if $j=2m$.}
 		\end{cases}
	\end{align*}	
	
Here $\p_0=1$.	Let $\Omega \subset \Xcan$ be the open dense subset where the coordinates $\p_m,\p_{m+1},\dots,\p_{2m-2}$ do not vanish. This subset is isomorphic to the Lusztig torus $\Xlus$ via the change of coordinates
	\begin{align*}
 		a_i = \frac{\p_{2m-1}\delta_{i}}{\p_{2m-1-i}\delta_{i-1}}, \qquad b_i = \frac{\p_{2m-1}}{\p_{2m-1-i}} \text{ for all $1 \leq i \leq m-1$,} \qquad 
 		c = \frac{\p_{m}^2}{\delta_{m-1}}. 
\end{align*}
	
	\begin{lemma}\label{l:factorisation}
 		For any element $(\p_0 :\p_1:\dots:\p_{2m-1}) \in \Omega$, $\Phi(\p_0 :\p_1:\dots:\p_{2m-1})$ factorizes as $u_1 \dot w_P \bar u_2$, where
 		\[
  			\bar u_2 = y_1(a_1)\dots y_{m-1}(a_{m-1}) y_m(c) y_{m-1}(b_{m-1}) \dots y_1(b_1)
 		\]
 		and $u_1$ is given by the matrix
 		\[
  			\begin{pmatrix}
  				1 & \frac{-(a_1+b_1)}{a_1 \dots a_{m-1} c b_{m-1} \dots b_1} & \dots & \frac{-(a_{m-1}+b_{m-1})}{a_1 \dots a_{m-1} c b_{m-1}} & \frac{-1}{a_1 \dots a_{m-1}} & \dots & \frac{-1}{a_1} & \frac{-1}{a_1 \dots a_{m-1} c b_{m-1} \dots b_1} \\
   				&&  &&&&&   \\
   				& 1 &&&&&& \frac{1}{a_1} \\
   				&& \ddots &&&&& \vdots \\
     			&&& 1 &&&& \frac{(-1)^{m}}{a_1 \dots a_{m-1}} \\
				&&  &&&&&   \\
  				&&&& 1 &&& (-1)^{m-1} \frac{a_{m-1}+b_{m-1}}{a_1 \dots a_{m-1} c b_{m-1}} \\
   				&&&&& \ddots && \vdots \\
   				&&&&&& 1 & \frac{-(a_1+b_1)}{a_1 \dots a_{m-1} c b_{m-1} \dots b_1} \\
   				&&&&&&& 1
  			\end{pmatrix}
 		\]
	\end{lemma}

	\begin{proof}[Proof of Lemma~\ref{l:factorisation}]
		Using the definition of the $y_i$, it is easy to check that $\bar u_2 \cdot v_j$ is equal to
 		\begin{align*}
 			\begin{cases}
  				v_j + \sum_{\ell=0}^{m-1-j} (a_{j+\ell} + b_{j+\ell}) \left(\prod_{r=0}^{\ell-1}b_r\right) v_{j+\ell+1} + \sum_{\ell=0}^{m-1} \left(\prod_{r=1}^\ell a_{m-r}\right)& c b_{m-1} \dots b_j v_{m+1+\ell} \\
  				 &\text{if $1\leq j\leq m-1$,} \\
  				v_m + \sum_{k=0}^{m-1} a_{m-k} \dots a_{m-1} c &\text{if $j=m$,} \\
 				v_j + ( a_{2m-j}+b_{2m-j} ) \sum_{\ell=0}^{2m-1-j} \left(\prod_{r=1}^{\ell-1} a_{2m-1-j-r}\right) v_{j+1+\ell} &\text{if $m+1 \leq j \leq 2m$.}
 			\end{cases}
 		\end{align*}
 		Now a straightforward computation shows that $\Phi(\p_0 :\p_1:\dots:\p_{2m-1}) = u_1 \dot w_P \bar u_2$.
	\end{proof}

	Lemma~\ref{l:factorisation} shows that $\Phi(\Omega)$ is contained in $Z_{q=1}:=B_- \dot w_0 \cap U_+  \dot w_P U_-$. In the next lemma we prove that the entire image of $\Phi$ is contained in $Z_{q=1}$.

	\begin{lemma} \label{l:image}
		The image $\Phi(\p_0 :\p_1:\dots:\p_{2m-1})$ of any element $(\p_0 :\p_1:\dots:\p_{2m-1}) \in \Xcan$ lies in $Z_{q=1}$.
	\end{lemma}

	\begin{proof}[Proof of Lemma~\ref{l:image}]
 		Since $\Omega$ is open dense in $\Xcan$ we have that $\Phi(\Xcan) \subset B_- \dot w_0 \cap \overline{U_+  \dot w_P U_-}$. Suppose indirectly there exists $x=(\p_0:\p_1:\dots:\p_{2m-1}) \in \Xcan$ such that $\Phi(x) \not \in U_+  \dot w_P U_-$. Then from Bruhat decomposition, we get $\Phi(x) \dot w_0^{-1} \in U_+ \dot w U_+$ with $w < w_P w_0$. It follows that we must have
		\begin{equation}\label{e:Bruhatzero}
			0= \langle \Phi(x)\dot w_0^{-1} \cdot v_{\omega_1}^+,v_{\omega_1}^- \rangle = \langle \Phi(x) \cdot v_{\omega_1}^-,v_{\omega_1}^- \rangle
		\end{equation}
		in the representation $V_{\omega_1}$ of $G$, where $v^+_{\omega_1}=v_1$ and $v^-_{\omega_1}=v_{2m}$, compare Section~\ref{SpNotation}. However we have seen that the lower right hand corner of the matrix representing $\Phi(x)$ is equal to $1$, contradicting~\eqref{e:Bruhatzero}. Thus we must have $\Phi(x) \in Z_{q=1}$.
	\end{proof}
	
	We have thus shown that the map $\Phi$ is a regular morphism $\Xcan \to Z_{q=1}$. Moreover by construction the composition with the coset map, $\pi: Z_{q=1}\to \Xlie, z\mapsto zB_-$, gives the map
	\[\Psi_{q=1}\inv:\Xcan\to\Xlie\]
	we were looking for. 
	This concludes the proof of Proposition~\ref{p:iso}.
\end{proof}

We have now proved our main comparison result, Theorem~\ref{t:canLie}.

\section{The Jacobi ring presentation of $\qH^*(X,\C)[q\inv]$}\label{s:Jacobi}

The main result of~\cite{rietsch} was to show that there is an isomorphism between the quantum cohomology of $X$ and the Jacobi ring of $(\Xlie,\Wlie)$ (for the case of the quadric $X$ as well as for a general homogeneous space $X=G^\vee/P^\vee$). This result made use of the remarkable Peterson presentation  \cite{peterson} of $\qH^*(G^\vee/P^\vee)$, which identifies 
the quantum cohomology ring with the coordinate ring of an associated affine variety  $\mathcal Y_P\subset G/B$. The variety  $\mathcal Y_P$ is called the `Peterson variety' associated to the parabolic subgroup $P$.

Now that we have proved that the canonical LG model $(\Xcan,\Wcan)$ is isomorphic to the Lie-theoretic LG model $(\Xlie,\Wlie)$, Theorem~\ref{t:canLie}, we may apply the result from~\cite{rietsch} to  deduce that there is an isomorphism between the quantum cohomology of $X=\Q_{2m-1}$ and the Jacobi ring of $(\Xcan,\Wcan)$,
\begin{equation*}
	\C\left[ \Xcan \times \C^*_q \right] / \left( \frac{\partial \Wcanq}{\partial \p_1}, \dots, \frac{\partial \Wcanq}{\partial \p_{2m-1}} \right)\cong\qH^*(X,\C)\left[q^{-1}\right] .
 \end{equation*}
In this section we prove that the isomorphism sends the homogeneous coordinate $\p_i$ to the Schubert class $\sigma_i$. 

\begin{theo}\label{t:Jacobi}
The isomorphism  
 	\begin{equation}\label{e:Jacobi}
  		\C[\Xcan \times \C^*_q] / (\partial \Wcanq) \to \qH^*(X)[q\inv]
 	\end{equation}
defined above, identifies the coordinate $q$ with the quantum parameter $q$ and sends $\p_i$ to the Schubert class $\sigma_i \in H^{2i}(X,\C)$.
\end{theo}

We first prove the following lemma. Note that when we write $p_0$ we mean $1$, since this homogeneous coordinate has been fixed as $p_0=1$.  
\begin{lemma}\label{l:q-rel}
In the Jacobi ring   		$\C[\Xcan \times \C^*_q]/ (\partial \Wcanq)$ of $\Wcan$, the element $\delta_0=p_0 p_{2m-1}$ has the property
\begin{equation}\label{e:q-rel1}
\delta_0^2=p_{2m-1}^2= q^2,\quad \text { and }\quad \delta_0 p_i=q p_i, \ \text{ for $i=1,\dotsc,2m-1$.} 
\end{equation}
For the elements $\delta_\ell$ with $1 \leq \ell \leq m-1$ we have 
\begin{equation}\label{e:q-rel2}
\delta_{\ell}=\p_\ell \p_{2m-1-\ell} - \p_{\ell-1} \p_{2m-\ell} + \dots + (-1)^l \p_0\p_{2m-1}=\begin{cases} q  & \text{ if $\ell$ is odd}\\
 \p_{2m-1} & \text{ if $\ell$ is even,}
\end{cases}
 	\end{equation}
in the Jacobi ring. 
\end{lemma}

\begin{proof} The equations~\eqref{e:q-rel1}  and \eqref{e:q-rel2} for $\p_i$ replaced by $\sigma_i$ are a straightforward consequence of quantum Schubert calculus on the quadric (which can be deduced from the quantum Chevalley formula \cite{FW} in this case), see also \eqref{e:quant-hyperplane}. 
It is not hard to check by a direct calculation that the relations 
\[
\frac{\partial\Wcanq}{\partial \p_i}=0
\]
imply that the $\p_i$  in the Jacobi ring satisfy all the same relations as the $\sigma_i$ do in the quantum cohomology ring.

Alternatively, particularly to prove \eqref{e:q-rel1}, one can check that the coordinates of the critical points of $\Wcanq$ satisfy the  equations. They are all non-degenerate; the quantum cohomology ring is semisimple. These  critical points are worked out explicitly in \cite[Proposition~2.3]{PRW:quadrics}. 
\end{proof}

\begin{proof}[Proof of Theorem~\ref{t:Jacobi}]
While we already noted in the preceding proof that the $\p_i$ satisfy the relations of quantum Schubert calculus in the Jacobi ring of $\Wcan$,
  the statement of the theorem is about a specific isomorphism. It remains to check that this isomorphism  does indeed send  $p_i$ to $\sigma_i$.
  
The isomorphism, 
\begin{equation}\label{e:QcohIso}
\qH^*(X,\C)[q\inv]\longrightarrow  \C[\Xlie \times \C^*_q] / (\partial \Wlieq),
\end{equation} 
coming from \cite{rietsch} and involving the Peterson presentation \cite{peterson} 
takes the following form. The Schubert class $\sigma_k$ is associated to the  element $w^{(k)}\in W^P$ defined by
\begin{align*}
 		w^{(k)} = \begin{cases}
        			s_k s_{k-1} \dots s_1 & \text{if $k \leq m$,} \\
         		s_{2m-k} \dots s_{m-1} s_m s_{m-1} \dots s_1 & \text{if $m+1 \leq k \leq 2m-1$.}
        		  \end{cases}
 	\end{align*}  
Note that $w^{(k)}=w_{(k)}\inv$, compare~Section~\ref{s:Xcan}. 
The  isomorphism \eqref{e:QcohIso} is explicitly given by 
$\sigma_k\mapsto f_k$ where $f_k$ is
%see  Section~\ref{s:Xcan}, 
the regular function
\[
f_k(zB_-):= \frac{\langle z\cdot v^-_{\omega_1},\dot w^{(k)}\cdot v_{\omega_1}^-\rangle}{\langle z\cdot v^-_{\omega_1}, v_{\omega_1}^-\rangle}.
\] 
Here $v^-_{\omega_1}=v_{2m}$ is the lowest weight vector of the representation $V=V_{\omega_1}$. We may assume that $z=u_1 t(q)\dot w_P\bar u_2$, as in Section~\ref{s:LieNotation}. We need to show that
\[
f_k(zB_-)=p_k(\bar u_2),
%=\frac{\langle v^0 \cdot \bar u_2,  v^0 \cdot \dot w_{(k)} \rangle}{\langle v^0 \cdot \bar u_2, v^0 \rangle}= \frac{\langle v^0 \cdot \bar u_2, v^k \rangle}{ \langle v^0 \cdot \bar u_2, v^0 \rangle}.
\]
whenever $P\bar u_2$ is a critical point of $\Wcanq$ with $q=\alpha_1(t)$.  
We use the notations from  the proof of Proposition~\ref{p:iso}. Recall that there we have a map $\Phi:\Xcan\to B_-\dot w_0$ given explicitly in terms of the coordinates $\p_i$, for which 
\[
t(q)\ \Phi(p_1(\bar u_2),\dotsc, \p_{2m-1}(\bar u_2))=z=u_1 t(q)\dot w_P\bar u_2 \in Z.
\]
We can  now work out $f_k(zB_-)$ by looking at the entries of the last column of the matrix $z=t(q)\Phi(p_1(\bar u_2),\dotsc, \p_{2m-1}(\bar u_2))$. Namely we get
\[
f_k(zB_-)=
\begin{cases}
1,&\text{if $k=0$},\\
\frac{q p_k}{\delta_{k-1}}, &\text{if $ 1\le k\le m$,}\\
\frac{q p_k}{\delta_{2m-k-1}}, & \text{ if $m+1\le k \le 2m-1$,}\\
\frac{q^2}{\delta_0},&\text{ if $k=2m-1$.}
\end{cases}
\]
Applying the relations \eqref{e:q-rel1} and \eqref{e:q-rel2} we get the identity $f_k(zB_-)=\p_k(\bar u_2)$ as required. 
\end{proof}

It is interesting to note that under the isomorphism from Theorem~\ref{t:Jacobi}, the summands of $\Wcan$ map to $p_1$ or $2p_1$ and $\Wcan$ maps to the anticanonical class of $X$ in the quantum cohomology ring. 
Namely for $k=2,\dotsc, m-2$ it follows from the relations that the $k$-th summand of $\Wcan=W_1+W_2+\dotsc + W_{m-1}+qW_m$ in the Jacobi ring simplifies to 
\[
W_k=\frac{\p_{k+1}\p_{2m-1- k}}{\delta_k} = 2\p_1,
\]
while $k=1,m-1, m$ gives 
\[ W_1=p_1, \quad W_{m-1}=\frac{\p_{m}^2}{\delta_m} = \p_1,
\quad W_{m}=q\frac{p_1}{p_{2m-1}}=p_1.
\] 
In total we have
$\Wcan=(2m-1) p_1$ in $\C[\Xcan\x\C^*_q]/(\partial{\Wcanq})$ and $\Wcan$ represents the anticanonical class $(2m-1) \sigma_1$ of $X=Q_{2m-1}$ via the isomorphism~\eqref{e:Jacobi}.

\section{Comparison with other LG models}\label{s:Comparison}

Let us now see how our canonical mirror $(\Xcan,\Wcan)$ and the corresponding Laurent polynomial mirror $(\Xlus,\Wlus)$ compare with previous Landau-Ginzburg models for odd quadrics presented in Section~\ref{s:LGoverview}. From Theorem~\ref{t:canLie} we already know that $(\Xcan,\Wcan)$ is isomorphic to the Lie-theoretic mirror $(\Xlie,\Wlie)$, so it only remains to consider the Givental, Przyjalkowski  mirrors and the partial compactification by Gorbounov and Smirnov.%, see~\cite{GS}.

\vskip .3cm
\paragraph{\bf Comparison with the Givental mirror and Przyjalkovski Laurent polynomial.}
First recall the definition of the Givental mirror of $\Q_{2m-1}$,
\begin{align*}
	\Xgiv &= \left\{ (\nu_1,\dots,\nu_{2m+1}) \in (\C^*)^{2m+1} \mid \prod_{i=1}^{2m+1} \nu_i = q, \nu_{2m}+\nu_{2m+1} = 1 \right\},\\
	 \Wgiv &= \nu_1 + \dots + \nu_{2m-1}.
\end{align*}
The Laurent polynomial `extension' of this mirror, written down in~\cite{przyjalkowski}, is as follows
\[
	\Xprz = (\C^*)^{2m-1},\quad \Wprz = Y_1+\dots+Y_{2m-2}+\frac{(Y_{2m-1}+q)^2}{Y_1\dots Y_{2m-1}}.
\]
It is obtained via the change of coordinates
\begin{align*}
	Y_i = 	\begin{cases}
				\nu_{i+1} &\text{if $1 \leq i \leq 2m-2$,} \\
				q \frac{\nu_{2m}}{\nu_{2m+1}} &\text{if $i=2m-1$.}
			\end{cases}
\end{align*}
The torus $\Xprz$ is slightly larger that $\Xgiv$, as
\[
	\Xgiv \cong \Xprz \setminus \{ Y_{2m-1}+q=0 \}.
\]
Our canonical mirror is a partial compactification of the Givental mirror, as shown by the following result, whose proof is immediate.

\begin{prop}
	The change of coordinates (dependant on $q$)
	\begin{align*}
		Y_i = 	\begin{cases}
					\frac{\p_i}{\p_{i-1}} &\text{ if $1 \leq i \leq m-1$,} \\
					\frac{\p_{2m-1-i}\delta_{2m-3-1}}{\p_{2m-2-i}\delta_{2m-2-i}} &\text{if $m \leq i \leq 2m-3$,} \\
					q\frac{\p_1}{\p_{2m-1}} &\text{if $i=2m-2$,} \\
					q\frac{\delta_{2m-2}}{\delta_{m-1}} &\text{if $i=2m-1$,}
				\end{cases}
	\end{align*}
	induces an isomorphism between $\Xprz$ and the torus inside $\Xcan$ where the Pl\"ucker coordinates $\p_i$ for $0 \leq i \leq m-1$ are all non-zero.
\end{prop}

Thus we have two distinguished tori inside the canonical mirror variety $\Xcan$, namely $\Xprz$ and the Lusztig torus $\Xlus$. They are however distinct, since $\Xlus$ is the torus inside $\Xcan$ where the Pl\"ucker coordinates $\p_i$ for $m \leq i \leq 2m-1$ are all non-zero. In fact, $\Xcan$ is an example of a cluster variety, see~\cite[Section 12]{GLS-Partial}, so it contains several tori.

\vskip .3cm
\paragraph{\bf Comparison with the Gorbounov-Smirnov mirror.} 
The Landau-Ginzburg model $(\XGS,\WGS)$ from~\cite{GS} goes as follows
\[
	\XGS = \{ (x;y_1,\dots,y_{m-1};z_1,\dots,z_{m-1}) \in \C^m \times (\C^*)^{m-1} \mid xy_1\dots y_{m-1}-1 \neq 0 \},
\]
\begin{equation*}
	 \WGS = \sum_{i=1}^{m-1} y_i(1+z_i) + q \frac{x^2}{(x y_1 y_2 \dots y_{m-1} - 1)z_1 z_2 \dots z_{m-1}}.
\end{equation*}
Consider the change of coordinates:
\begin{align*}
 	& y_i = \frac{\p_i}{\p_{i-1}} \qquad\forall\; 1 \leq i \leq m-1 ; \\
 	& z_1 = q\frac{\p_0}{\p_{2m-1}} ; \qquad z_i = \frac{\delta_{i-2}}{\delta_{i-1}} \qquad\forall\; 2 \leq i \leq m-1 ; \\
 	& x = \frac{\p_0\p_m}{\delta_{m-1}}.
\end{align*}

\begin{prop}\label{p:Comparison}
	The change of coordinates $ \{ \p_i \}\mapsto\{ x,y_i,z_i \} $ above defines an isomorphism between the torus  $\{ \p_1 \dots \p_{m-1} \neq 0 \}$ and $\{ y_1 \dots y_{m-1} \neq 0 \}$ inside $\XGS$ and the torus inside $\Xcan$, which pulls back the Gorbounov-Smirnov superpotential $\WGS$ to $\Wcan$. Moreover when $m=2$ the change of coordinates is a well-defined isomorphism between $\Xcan$ and $\XGS$.
\end{prop}

\begin{proof}
	We have $y_1(1+z_1) = \p_1 + q\frac{\p_1}{\p_{2m-1}}$, and $y_i (1+z_i) = \frac{\p_i \p_{2m-i}}{\delta_{i-1}}$ for $2 \leq i \leq m-1$. Moreover
 \[
  		x y_1 \dots y_{m-1}-1 = \frac{\delta_{m-2}}{\delta_{m-1}},\qquad
  		z_1 \dots z_{m-1} = \frac{q}{\delta_{m-2}},\qquad
  		x^2 = \frac{\p_m^2}{\left( \delta_{m-1} \right)^2},
 	\]
 	which gives
 	\[
  		q \frac{x^2}{(x y_1 y_2 \dots y_{m-1} - 1)z_1 z_2 \dots z_{m-1}} = \frac{\p_m^2}{\delta_{m-1}},
	 \]
 	hence the change of coordinates maps $\WGS$ to $\Wcan$. Finally, the change of coordinates is well-defined on the tori, and so is its inverse,
 \begin{align*}
 	&	\p_0=1,\qquad \p_i = \prod_{j=1}^{m-1} y_j \ \text{ for } 1 \leq i \leq m-1,\qquad \p_m=\frac{q}{z_1\dots z_{m-1}}\frac{x}{xy_1\dots y_{m-1}-1} \\
 &		\p_{2m-1-i} = \frac{q(1+z_{i+1})}{y_1 z_1 \dots y_i z_i}\  \text{ for } 1 \leq i \leq m-2,\qquad \p_{2m-1}=\frac{q}{z_1},
 	\end{align*}
	which concludes the proof.
\end{proof}
It follows from Proposition~\ref{p:Comparison} that the LG models $(\Xcan,\Wcan)$ and $(\XGS,\WGS)$ are isomorphic in the case of the three-dimensional quadric $\Q_3$, but in the general case we only get a birational equivalence.

\section{The $A$-model and $B$-model connections}\label{s:consequences}

Recall that we have proved that the canonical LG model $(\Xcan,\Wcan)$ is isomorphic to the Lie-theoretic LG model $(\Xlie,\Wlie)$; hence using~\cite{rietsch} we deduced that there is an isomorphism between the quantum cohomology of $X=\Q_{2m-1}$ and the Jacobi ring of $(\Xcan,\Wcan)$, 
\[
	\qH^*(X,\C)\left[q^{-1}\right] \cong \C\left[ \Xcan \times \C^*_q \right] / \left( \frac{\partial \Wcanq}{\partial \p_1}, \dots, \frac{\partial \Wcanq}{\partial \p_{2m-1}} \right).
\]
Furthermore we have proved that the isomorphism is given by mapping the Schubert class $\sigma_i \in H^{2i}(X,\Z)$ to the Pl\"ucker coordinate $\p_i$.

We may now prove a more detailed mirror theorem by comparing two flat connections, one related to $X=\Q_{2m-1}$ and one constructed from $(\Xcan,\Wcan)$. Let $\HH_A$ be the sheaf of regular functions of the trivial vector bundle with fibre $H^*(X,\C)$ over $\C_\hbar^* \times \C_q^*$, the two-dimensional complex torus with coordinates $\hbar$ and $q$. The \emph{A-model connection} ${}^A\nabla$, also known as the Dubrovin connection, is defined on $\HH_A$ by
\begin{align*}
	{}^A \nabla_{q \partial_q} &= q \frac{\partial}{\partial q} + \frac{1}{\hbar} p_1 \star_q \bullet, \\
 	{}^A \nabla_{ \hbar \partial_\hbar} &= \hbar \frac{\partial}{\partial \hbar} + \mathrm{gr} - \frac{1}{\hbar} c_1(TX) \star_q \bullet,
\end{align*}
where $\mathrm{gr}$ is a diagonal operator on $H^*(X,\C)$ given by $\mathrm{gr}(\alpha)=k\alpha$ for $\alpha \in H^{2k}(X,\C)$. Here we are using the conventions of \cite{iritani}. Let $\HH_A^\vee$ be the vector bundle on $\C_\hbar^*\x\C_q^*$ defined by $\HH_A^\vee=j^*\HH_A$ for $j:(\hbar,q)\mapsto (-\hbar, q)$. This vector bundle with the pulled back connection ${}^A\nabla^\vee=j^*\left({ }^A{\nabla}\right)$ is dual to $(\HH_A,{}^A\nabla)$ via the flat non-degenerate pairing,
\[
 	\left< \sigma_j,\sigma_k\right>= (2\pi i \hbar)^{2m-1} \int_{[X]} \sigma_j\cup \sigma_k=(2\pi i \hbar)^{2m-1} \delta_{j+k,2m-1}.
\]
The dual A-model connection ${}^A \nabla^\vee$ defines a system of differential equations which we call the (small) quantum differential equations
\begin{equation}\label{e:qde}
 	{}^A \nabla^\vee_{q \partial_q} S = 0. 
\end{equation}

Let us now define a $\C[\hbar^{\pm 1},q^{\pm 1}]$-module
\[
 	\mathcal G := \Omega^N\left(\Xcan\right)[\hbar^{\pm 1},q^{\pm 1}]/\left( d - \frac{1}{\hbar}d\Wcanq \wedge \bullet\right) \Omega^{N-1}\left(\Xcan\right)[\hbar^{\pm 1},q^{\pm 1}],
\]
where $\Omega^k\left(\Xcan\right)$ is the space of algebraic $k$-forms on $\Xcan$. We denote by $\HH_B$ the sheaf with global sections $\mathcal G$. 
The $B$-model connection, or Gauss-Manin connection, on $\HH_B$ is given by
\begin{align*}
 	{}^B \nabla_{q\partial_q} [\eta] &= q \frac{\partial}{\partial q} [\eta] + \frac{1}{\hbar} \left[ q\frac{\partial \Wcanq}{\partial q} \eta \right], \\
 	{}^B \nabla_{\hbar\partial_\hbar} [\eta] &= \hbar\frac{\partial}{\partial \hbar} [\eta] - \frac{1}{\hbar} \left[ \Wcanq \eta \right].
\end{align*}
Since $\Xcan$ is a cluster variety there exists a unique up to a scalar non-vanishing $N$-form on $\Xcan$ with simple poles along the boundary, which we denote by $\omegacan$, compare~\cite{LamSpeyerI}. Explicitly, in terms of homogeneous coordinates we write
\[
	\omegacan := \bigwedge_{i=1}^{m-1}\frac{ d\p_i}{\p_i} \wedge \bigwedge_{i=1}^{m-1} \frac{d\delta_i}{\delta_i} \wedge \frac{d\p_{2m-1}}{\p_{2m-1}}.
\]
\begin{theo}\label{t:connections}
The map
\begin{equation*}
\begin{array}{cccc}
	\Psi :& (\HH_A,{}^A\nabla)& \to &(\HH_B,{}^B\nabla)\\ 
	&\sigma_j &\mapsto & [\p_j\omegacan]
\end{array}
\end{equation*}
is an injective homomorphism of bundles with connection. Moreover, in the case of the three-dimensional quadric $\Q_3$ it is an isomorphism.
\end{theo}

\begin{proof}
We use the cluster variety structure of the mirror of $\Q_{2m-1}$. Namely, the coordinate ring $\C[\Xcan]$ has a cluster algebra structure of type $A_1^{m-1}$, which is described in detail in \cite[Section 2]{GLS-Survey} and \cite[Section 12]{GLS-Partial}. Consider the following initial quiver:
	\begin{center}\label{fig:clusterquiver}
		\begin{tikzpicture}[every node/.style={thick,circle,inner sep=1pt,minimum size=30pt,scale=0.66},scale=0.66]
			\node[draw=black] (0,0) (1){$\p_1$};
			\node[draw=black,right=of 1] (2){$\p_2$};
			\node[draw=none,right=of 2] (3){$\dots$};
			\node[draw=black,right=of 3] (4){$\p_{m-2}$};
			\node[draw=black,right=of 4] (4bis){$\p_{m-1}$};
			\node[draw=black,below=of 1] (6){$\delta_1$};
			\node[draw=black,left=of 6] (5){$p_{2m-1}$};
			\node[draw=black,right=of 6] (7){$\delta_2$};
			\node[draw=none,right=of 7] (8){$\dots$};
			\node[draw=black,right=of 8] (9){$\delta_{m-2}$};
			\node[draw=black,right=of 9] (10){$\delta_{m-1}$};
 	    	\path (5) edge[->,thick] (1);
 	    	\path (1) edge[->,thick] (6);
 	    	\path (6) edge[->,thick] (2);
  			\path (2) edge[->,thick] (7);
  	    	\path (7) edge[->,thick] (3);
  	    	\path (8) edge[->,thick] (4);
 	    	\path (4) edge[->,thick] (9);
 	    	\path (9) edge[->,thick] (4bis);
    		\path (4bis) edge[->,thick] (10);
 		\end{tikzpicture}
	\end{center}
	Here the initial cluster variables correspond to the vertices in the top row of the quiver, while the frozen variables (or coefficients) correspond to the vertices in the bottom row. In particular, it is of finite type, and there are $2^{m-1}$ different clusters, consisting of
	\begin{itemize}
		\item the cluster variables $r_1,\dots,r_{m-1}$, where $r_i \in \{\p_i,\p_{2m-1-i}\}$;
		\item the frozen variables (or coefficients) $\delta_1,\dots,\delta_{m-1}$, and $\p_{2m-1}$.
	\end{itemize}
	Moreover we have set $\p_0=1$. The exchange relations are
	\begin{align}\label{e:mutations}
		\p_i \p_{2m-1-i} = 	\begin{cases}
								\p_{2m-1} + \delta_1 & \text{for $i=1$;}\\
								\delta_{i-1}+\delta_i & \text{for $2 \leq i \leq m-1$.}
							\end{cases}
	\end{align}
	We need to prove that $\Psi$ maps the $A$-model connection to the $B$-model connection. We use a change of coordinates to reduce the problem to checking only the action of $q\partial_q$. Namely, this follows by replacing $(\p_i,q,\hbar)$ with $(\mathbf p_i,\mathbf q,\hbar)$, where 
	\[
		\mathbf p_i=\hbar^{-i}\p_i,\qquad \mathbf q=\hbar^{1-2m}q, \qquad \hbar=\hbar,
	\]
	and observing that written in these coordinates the Gauss-Manin system for $\frac 1\hbar \Wcanq$ no longer involves the $\hbar$.  

	Now we check that the map $\Psi$ preserves the action of $q\partial_q$. We consider the following identities in $\qH^*(\Q_{2m-1},\C)$, which are a special case of results in \cite{FW}:
\begin{align}\label{e:quant-hyperplane}
	\sigma_1 \star_q \sigma_i = 	\begin{cases}
										\sigma_{i+1} & \text{for $0 \leq i \leq m-2$ or $m \leq i \leq 2m-3$;} \\
										2\sigma_m & \text{for $i=m-1$;} \\
										\sigma_{2m-1} + q & \text{for $i=2m-2$;} \\
										q \sigma_1 & \text{for $i=2m-1$.}
								  \end{cases}
\end{align}	
	We need to prove that there are similar identities on the $B$ side:
	\begin{align}\label{e:identities-W}
		\left[q\frac{\partial \Wcanq}{\partial q} \p_i \omegacan\right] = 	\begin{cases}
																	[\p_{i+1} \omegacan] & \text{for $0 \leq i \leq m-2$ or $m \leq i \leq 2m-3$;} \\
																	[2\p_{m-1} \omegacan] & \text{for $i=m-1$;} \\
																	[(\p_{2m-1} + q) \omegacan] & \text{for $i=2m-2$;} \\
																	[q \p_1 \omegacan] & \text{for $i=2m-1$,}
							 									\end{cases}
	\end{align}
	where $\omegacan$ is the canonical $(2m-1)$-form on $\Xcan$. The proof of these identities on the $B$ side proceeds by constructing closed $(2m-2)$-forms $\nu_i$ such that the relation corresponding to $p_i$ will follow from the fact that 
	\begin{equation*}
		[d \Wcanq \wedge \nu_i]=[( d +  d\Wcanq \wedge - )\nu_i]= 0.
	\end{equation*}
	(The first equality above comes from the fact that $\nu_i$ is closed, and the second comes from the definition of the $B$-model.)

	Concretely, we will pick a cluster $\mathcal C$ containing a particular Pl\"ucker coordinate, say $p_i$, and use the following Ansatz for constructing $\nu_i$. We define a vector field
	\begin{equation*}
		\xi_i= p_i \left(\sum_{c\in\mathcal C\setminus\{p_i\}} m_c c\partial_c\right) 
	\end{equation*}  
	and an associated $(2m-2)$-form by insertion $\nu_i=\iota_{\xi_i}\omegacan$. Here the $m_c$'s are constants and $\iota$ is the interior product. 
	
	\begin{lemma}\label{l:nu}
		The forms $\nu_i$ are closed regular $(2m-2)$-forms of $\Xcan$.
	\end{lemma}
	
	\begin{proof}[Proof of Lemma~\ref{l:nu}]
		 On the torus $T_\mathcal{C}$ associated with the cluster $\mathcal{C}$, the canonical $(2m-1)$-form $\omegacan$ may be written as
		 \[
		 	\omegacan = \bigwedge_{p \in \mathcal C} \frac{dp}{p}.
		 \] 
		 For $c\in \mathcal C$, we have $\iota_{c \partial_c} \omegacan = \bigwedge_{p \in \mathcal C \setminus \{c\}} \frac{dp}{p},$ and so $\nu_i$ is a $\C$-linear combination of terms of the form $p_i \bigwedge_{p\in \mathcal C \setminus \{c\}} \frac{dp}{p}$ for $c \neq p_i$. Since $p_i$ lies in $\mathcal C \setminus \{c\}$, such a term is closed, hence so is $\nu_i$.
		 
		 Moreover, $\nu_i$ is regular on $T_\mathcal{C}$, and we will show it is also regular on $T_\mathcal{C}'$ for any cluster $\mathcal{C}'$ obtained from $\mathcal{C}$ by a single mutation. Indeed, $\nu_i$ is a linear combination of terms of the form 
		 \[
		 	 \p_i \bigwedge_{p \in \mathcal{C} \setminus \{\p_j\}} \frac{dp}{p}
		 \]
		 for $0 \leq i \neq j \leq 2m-1$. If $\mathcal{C}'$ is the cluster obtained by mutating some $p \in \mathcal{C}$, differentiating the exchange relation gives an identity of the form $pdq+qdp=0$, hence
		 \[
		 	\p_i \bigwedge_{p \in \mathcal{C} \setminus \{\p_j\}} \frac{dp}{p} = \pm \p_i \bigwedge_{p \in \mathcal{C}' \setminus \{\p_j\}} \frac{dp}{p}.
		 \]
		 The right-hand side is regular on $T_{\mathcal{C}'}$ as claimed.
		 
		 Thus we have proved that $\nu_i$ is regular on all cluster tori which are `adjacent' to $T_\mathcal{C}$. Using~\cite[Proposition 9.6]{LamSpeyerI} we conclude that $\nu_i$ is regular on the whole of $\Xcan$, which concludes the proof of the lemma.
	\end{proof}
	
	Now since $d \Wcanq \wedge \omegacan = 0$, we have that $d \Wcanq \wedge \nu_i = \pm d \Wcanq(\xi_i) \omegacan$. It follows that 
	\begin{equation*}
		d \Wcanq \wedge \nu_i=  p_i \left(\sum_{c\in\mathcal C\setminus\{p_i\}} m_c c \frac{\partial \Wcanq}{\partial c} \right)\omegacan.  
	\end{equation*}
	Therefore e.g. in order to prove that $\left[q\frac{\partial \Wcanq}{\partial q} \p_i \omegacan\right]-[p_{i+1} \omegacan] = 0$, we will show that $q\frac{\partial \Wcanq}{\partial q} \p_i  - p_{i+1}$ has the form  $p_i \left(\sum_{c\in\mathcal C\setminus\{p_i\}} m_c c \frac{\partial \Wcanq}{\partial c} \right) $, for some choice of coefficients $m_c$.
  
	To prove these identities, we will work with two clusters:
	\begin{itemize}
		\item the initial cluster $\mathcal{C}_1 = \{ p_1,\dots,p_{m-1},\delta_1,\dots,\delta_{m-1}, \p_{2m-1} \}$;
		\item the cluster $\mathcal{C}_2 = \{ p_{2m-2},\dots,p_m,\delta_1,\dots,\delta_{m-1}, \p_{2m-1}\}$.
	\end{itemize}

	Let us first start with $\mathcal{C}_1$ and express $\Wcanq$ in terms of it using the exchange relations \eqref{e:mutations}. To simplify our notation we let $\delta_0$ denote $\p_{2m-1}$.
	\[
		\Wcanq =  \p_1  + \sum_{\ell=1}^{m-1} \left( \frac{\p_{\ell+1}\delta_{\ell-1}}{\p_\ell \delta_\ell} + \frac{\p_{\ell+1}}{\p_\ell} \right) + q \frac{\p_1}{\delta_0}.
	\]
	The partial derivatives of $\Wcanq$ are:
	\begin{align*}
		q\frac{\partial \Wcanq}{\partial q} &= q\frac{p_1}{\delta_0},\\
		\p_1 \frac{\partial \Wcanq}{\partial \p_1} &= \p_1-\frac{\p_2\delta_0}{\p_1\delta_1}-\frac{\p_2}{\p_1}+q\frac{\p_1}{\delta_0}, \\
		\p_i \frac{\partial \Wcanq}{\partial \p_i} &= \frac{\p_i\delta_{i-2}}{\p_{i-1}\delta_{i-1}}+\frac{\p_i}{\p_{i-1}}-\frac{\p_{i+1}\delta_{i-1}}{\p_i\delta_i}-\frac{\p_{i+1}}{\p_i} \text{ for $2 \leq i \leq m-1$}, \\
		\delta_0 \frac{\partial \Wcanq}{\partial \delta_0} &= \frac{\p_2\delta_0}{\p_1\delta_1} - q\frac{\p_1}{\delta_0}, \\
		\delta_i \frac{\partial \Wcanq}{\partial \delta_i} &= -\frac{p_{i+1}\delta_{i-1}}{\p_i\delta_i} + \frac{\p_{i+2}\delta_i}{\p_{i+1}\delta_{i+1}} \text{ for $1 \leq i \leq m-1$}.
	\end{align*}
	Hence 
	\[
		q \frac{\partial \Wcanq}{\partial q} p_{i} - p_{i+1} =  -\p_{i} \left(
\sum_{j=i+1}^{m-1} \p_j \frac{\partial \Wcanq}{\partial \p_j} + \sum_{j=0}^{m-1} \delta_j \frac{\partial \Wcanq}{\partial \delta_j} + \sum_{j=i}^{m-1} \delta_j \frac{\partial \Wcanq}{\partial \delta_j} \right)
	\]
	for $0 \leq i \leq m-2$, and
	\[
		q \frac{\partial \Wcanq}{\partial q} p_{m-1} - 2\p_m =  
 -\p_{m-1} \left( \sum_{j=0}^{m-1} \delta_j \frac{\partial \Wcanq}{\partial \delta_j} \right).
	\]
	Since the right-hand sides of the above equations have the form $p_i \left(\sum_{c\in\mathcal C\setminus\{p_i\}} m_c c\partial_c \Wcanq\right) $, this proves identity \eqref{e:identities-W} for $0 \leq i \leq m-2$.

	To prove the remaining identities, we use the cluster  $\mathcal{C}_2$. In this cluster chart, $\Wcanq$ takes the following form:
	\begin{align*}
		\Wcanq = & \frac{\delta_0}{\p_{2m-2}} + \frac{\delta_1}{\p_{2m-2}} + \sum_{\ell=1}^{m-1} \left( \frac{\p_{2m-1-\ell}}{\p_{2m-2-\ell}} + \frac{\p_{2m-1-\ell}\delta_{\ell+1}}{\p_{2m-2-\ell}\delta_\ell}\right) + \frac{q}{\p_{2m-2}} + \frac{q\delta_1}{\p_{2m-2}\delta_0}.
	\end{align*}
	Working out the partial derivatives of $\Wcanq$ as before, we get
	\[
		q \frac{\partial \Wcanq}{\partial q} p_i -p_{i+1} = p_i \left(
-\sum_{j=i+1}^{2m-2} \p_j\frac{\partial \Wcanq}{\partial \p_j} - \sum_{j=0}^{2m-2-i} \delta_j \frac{\partial \Wcanq}{\partial \delta_j} \right)\text{for $m \leq i \leq 2m-3$.}
	\]
	Recall that $\delta_0$ is $p_{2m-1}$. The final two relations are
	\begin{align*}
 		q\frac{\partial \Wcanq}{\partial q} \p_{2m-2} - (p_{2m-1}+q) &= -p_{2m-2}\delta_0 \frac{\partial \Wcanq}{\partial \delta_0}
 \text{\quad and } \\
		q \frac{\partial \Wcanq}{\partial q}\p_{2m-1} - q p_1 &=0.
	\end{align*}
	This gives us the identities \eqref{e:identities-W} for $m-1 \leq i \leq 2m-1$, which concludes the proof of the homomorphism. To deduce that the map is injective as claimed we  use the same strategy as in \cite[Lemma~9.3]{MR}. Namely, observe that the relations in the Gauss-Manin system recover the relations of the Jacobi ring as $\hbar$ tends to zero. On the other hand as we already proved, the Jacobi ring is isomorphic to quantum cohomology with the homogeneous coordinates $\p_i$ playing the role of the Schubert basis. Therefore we see that the $[\p_i\omegacan]\in \mathcal G$ are linearly independent in the $\hbar \to 0$ limit. Hence they must be linearly independent already in $\mathcal G$. 
	
In the case of $Q_3$ the added surjectivity result is a consequence of the fact that $(\Xcan,\Wcan)$ is isomorphic to the Gorbounov-Smirnov mirror in that case, see Proposition~\ref{p:Comparison}.  Indeed $\WGS$ is cohomologically tame \cite{GS}, hence so is $\Wcan$. Therefore $\mathcal G$ is a free $\C[\hbar^{\pm 1},q^{\pm 1}]$-module of rank $2m$ (cf. \cite{sabbah}), and $\HH_B$ a trivial vector bundle of that dimension. 
	\end{proof}

Let $\Gamma_0$ be a compact oriented real $(2m-1)$-dimensional submanifold of $\Xcan$ representing a cycle in $H^{2m-1}(\Xcan,\Z)$ dual to $\omegacan$, in the sense that $\frac{1}{(2 i \pi)^{2m-1}} \int_{\Gamma_0} \omegacan=1$. From Theorem~\ref{t:connections} we deduce the following formula. 
\begin{cor}\label{c:int}
	The integral formula
 	\[
  		S_0(\hbar,q) = \frac{1}{(2i \pi \hbar)^{2m-1}}\sum_{j=0}^{2m-1}\left( \int_{\Gamma_0} e^{\frac{\Wcanq}{\hbar}} p_j\omegacan\right) \sigma_{2m-1-j}
 	\]
	describes a solution to the quantum differential equation \eqref{e:qde}.  \qed
\end{cor}
The corollary follows as in \cite[Theorem~4.2]{MR}. If we replace $(\Xcan,\Wcan)$ by the isomorphic LG model $(\Xlie,\Wlie)$ the above corollary implies a special case of
 \cite[Conj. 8.1]{rietsch} for odd-dimensional quadrics.

\section{The mirror to $\Q_3$}

In this section we work out in detail the example of the three-dimensional quadric, $\Q_3$, to illustrate our main results. 

\vskip.2cm
\paragraph{\bf The Laurent polynomial mirror $(\Xlus,\Wlus)$.}
Recall from Section~\ref{s:Lie} the definition of the variety $Z \subset G$,
\[
	Z := B_- \dot w_0 \cap U_+ T^{W_P} \dot w_P U_-.
\]
A generic element $g \in Z$ can be written as $g = u_1 t \dot w_P \bar u_2$, where 
\[
  \bar u_2 = y_1(a)y_2(c) y_1(b),
\]
and $a,b,c$ are non-zero, i.e.
\[
 	\bar u_2=	\begin{pmatrix}
      				1         & 0     & 0        & 0 \\
      				a + b & 1     & 0        & 0 \\
     	 			c b     & c     & 1        & 0 \\
     	 			a c b & a c & a +b & 1
     			\end{pmatrix}.
\]
From this the expression of the Laurent polynomial mirror follows, namely,
\[
	\Xlus=(\C^*)^3_{a,b,c}, \Wlus=a+b+c+\frac{a+b}{abc}.
\]
This illustrates Theorem~\ref{t:LaurentLie} in the case of $\Q_3$.

\vskip.2cm
\paragraph{\bf The canonical mirror.}
The map $Z_t \to \X \cong\C\P^3$ takes $z=u_1 t \dot w_P\bar u_2$ to $P z=P \bar u_2$. This may be interpreted as taking $z$ to the span of the reverse row vector corresponding to the last row of $\bar u_2$ after the identification  $\X \cong\C\P^3$. The homogeneous coordinates of $\bar u_2$ are given by $\p_0=1, \p_1=a+b, \p_2=ac, \p_3=a c b$.

The image of $Z_{q=q_0}$ in $\C\P^3$ is independent of $q_0$, so we may choose $q_0$ to be $1$, and restrict our attention to $Z_{q=1}:=B_-\dot w_0\cap U_+\dot w_P U_-$. The image of $Z_{q=1}$ is obtained in coordinates $(\p_0:\p_1:\p_2:\p_3)$ by removing the anticanonical divisor 
\[
	D := \{\p_0 =0\}\cup\{\p_3 \p_0 - \p_2 \p_1=0\}\cup \{\p_3=0\}.
\]
Thus
\[
	\Xcan = \{ (\p_0:\p_1:\p_2:\p_3) \in\C\P^3 \mid \p_0(\p_1\p_2-\p_0\p_3)\p_3 \neq 0 \}
\]
In terms of the homogeneous coordinates, the canonical superpotential from Equation~\eqref{eq:W2} is given by
\[
 	\Wcanq = \frac{\p_1}{\p_0}+ \frac{\p_2^2}{\p_1 \p_2 - \p_0 \p_3}+ q \frac{\p_1}{\p_3}.
\]

\vskip .2cm
\paragraph{\bf Comparison with the Gorbounov-Smirnov-mirror.}
The mirror $(\XGS,\WGS)$ for $\Q_3$ is
\[
 	\XGS= \{ (x,y,z) \in \C^3 \mid (xy-1)z \neq 0 \} , \WGS = y(1+z) + q \frac{x^2}{(x y - 1)z}.
\]
It corresponds to $(\Xcan,\Wcan)$ via the change of coordinates (setting $\p_0=1$):
\[
 	x = \frac{\p_2}{\p_1 \p_2 - \p_3} ;\quad y =\p_1 ;\quad z = \frac{q}{\p_3}.
\]
This change of coordinates is well-defined on $\Xcan$, and its inverse,
\[
	\p_1=y ;\quad \p_2 = \frac{qx}{(xy-1)z};\quad \p_3 = \frac{q}{z}
\]
is well-defined on $\XGS$. This illustrates the result of Proposition~\ref{p:Comparison}, namely that for $\Q_3$ the LG models $(\Xcan,\Wcan)$ and $(\XGS,\WGS)$ are isomorphic, even though that is not the case in higher dimension.

\vskip .2cm
\paragraph{\bf Isomorphism between the quantum cohomology and the Jacobi ring.}
Recall that the cohomology of $\Q_3$ is generated by the Schubert classes $\sigma_i \in H^{2i}(\Q_3,\Z)$ for $i=0,\dots,3$. Moreover it has a presentation:
\[
	H^*(\Q_3,\Z)=\Z[\sigma_1,\sigma_2]/(\sigma_1^4,\sigma_1^2-2\sigma_2),
\]
and the quantum cohomology of $\Q_3$ is presented as follows,
\[
	\qH^*(\Q_3,\C)=\C[\sigma_1,\sigma_2,q]/(\sigma_1^4-q\sigma_1,\sigma_1^2-2\sigma_2).
\]
Now the Jacobi ring of $(\Xcan,\Wcan)$ is 
\[
	\C[\Xcan \times \C_q^*] / \left( \p_3^2-q\p_3,\p_1\p_2-2\p_3 ,\p_2^2-q\p_1\right)
\]
The map 
\[
	\p_1 \mapsto \sigma_1,\quad \p_2 \mapsto \sigma_2,\quad \p_3 \mapsto \frac{1}{2}\sigma_1^3
\]
defines an isomorphism between $\qH^*(\Q_3,\C)$ and the Jacobi ring of $(\Xcan,\Wcan)$ as in Theorem~\ref{t:Jacobi}.

\vskip .2cm
\paragraph{\bf The quantum differential equations.}
Recall from Section~\ref{s:consequences} that the dual $A$-model connection ${}^A \nabla^\vee$ defines a system of differential equations called the quantum differential equations,
\[
	{}^A \nabla^\vee_{q \partial q} S = 0.
\]
In the case of $\Q_3$ our mirror result, Theorem~\ref{t:connections}, tells us that the map $(\HH_A,{}^A \nabla) \to (\HH_B,{}^B \nabla)$ given by
\[
	\sigma_i \mapsto [\p_i \omegacan],
\]
where $\omegacan=\frac{d\p_1}{\p_1} \wedge \frac{d\delta_1}{\delta_1} \wedge \frac{d\p_3}{p_3}$,
is an isomorphism. If $\Gamma$ is a real $3$-dimensional cycle in $\C\P^3 \setminus D$ this implies in particular that the function
\begin{multline*}
	S_\Gamma(q) := \left( \int_\Gamma e^{\Wcanq} \p_3\ \omegacan \right) \sigma_0 +
	\left( \int_\Gamma e^{\Wcanq} \p_2\ \omegacan \right) \sigma_1 + 
	\left( \int_\Gamma e^{\Wcanq} \p_1\  \omegacan \right) \sigma_2\\ + \left( \int_\Gamma e^{\Wcanq}\  \omegacan \right) \sigma_3
\end{multline*}
satisfies ${}^A \nabla^\vee_{q \partial q} S_\Gamma(q)=0$.

%bibliography


\begin{thebibliography}{EHX97}

\bibitem[Bea95]{BeauvilleQCoh}
Arnaud Beauville.
\newblock Quantum cohomology of complete intersections.
\newblock {\em Mat. Fiz. Anal. Geom.}, 3--4(2):384--398, 1995.

\bibitem[BZ97]{BeZel:TotPos}
Arkady Berenstein and Andrei Zelevinsky.
\newblock Total positivity in {S}chubert varieties.
\newblock {\em Comment. Math. Helv.}, 72(1):128--166, 1997.
	
\bibitem[CMP07]{CMP:II}
Chaput, P.-E., L. Manivel, and N. Perrin. \newblock Quantum Cohomology of Minuscule Homogeneous Spaces II Hidden Symmetries, \newblock International Mathematics Research Notices, Vol. 2007, Article ID rnm107, 29 pages. doi:10.1093/imrn/rnm107.
	
	
\bibitem[CK99]{CoxKatz}
David A. Cox and Sheldon Katz.
\newblock Mirror symmetry and algebraic geometry.
\newblock {\em Mathematical Surveys and Monographs}, 68:xxii+469, American Mathematical Society, Providence, RI, 1999.

\bibitem[Dub96]{Dub:2DTFT}
Boris Dubrovin.
\newblock Geometry of {$2$}{D} topological field theories.
\newblock In {\em Integrable systems and quantum groups ({M}ontecatini {T}erme,
              1993)}, volume 1620 of {\em Lecture Notes in Math.}, pages 120--348. Springer, Berlin, 1996.

\bibitem[EHX97]{EHX}
Tohru Eguchi, Kentaro Hori, and Chuan-Sheng Xiong.
\newblock Gravitational quantum cohomology.
\newblock {\em Internat. J. Modern Phys. A}, 12(9):1743--1782, 1997.

\bibitem[FW04]{FW}
W.~Fulton and C.~Woodward.
\newblock On the quantum product of {S}chubert classes.
\newblock {\em J. Algebraic Geom.}, 13(4):641--661, 2004.

\bibitem[Gin95]{Gin:GS}
V.~Ginzburg.
\newblock {Perverse sheaves on a Loop group and Langlands' duality}.
\newblock arXiv:9511007, 1995.

\bibitem[Giv95]{Givental:ICM}
Alexander Givental.
\newblock {Homological geometry and mirror symmetry}.
\newblock  In {\em Proceedings of the International Congress of Mathematicians}, Vol. 1,  2 (Z\"urich, 1994), pages 472--480, Birkh\"auser, Basel, 1995. 

\bibitem[Giv96]{Givental:EqGW}
Alexander Givental.
\newblock Equivariant {G}romov-{W}itten invariants.
\newblock {\em Internat. Math. Res. Notices} 13:613--663, 1996.

\bibitem[GLS08a]{GLS-Partial}
Christof Gei\ss, Bernard Leclerc, and Jan Schr\"oer.
\newblock Partial flag varieties and preprojective algebras. 
\newblock {\em Ann. Inst. Fourier (Grenoble)} 58(3):825--876, 2008.

\bibitem[GLS08b]{GLS-Survey}
Christof Gei\ss, Bernard Leclerc, and Jan Schr\"oer.
\newblock Preprojective algebras and cluster algebras, Trends in representation theory
of algebras and related topics. 
\newblock In {\em EMS Ser. Congr. Rep., Eur. Math. Soc., Z\"urich} 253--283, 2008.

\bibitem[GS13]{GS}
Vassily Gorbounov and Maxim Smirnov.
\newblock {Some remarks on Landau-Ginzburg potentials for odd-dimensional
  quadrics}.
\newblock
Glasgow Mathematical Journal, Volume 57, Issue 3, DOI: http://dx.doi.org/10.1017/S0017089514000433, pp. 481-507, September 2015.

\bibitem[HV00]{HV}
Kentaro Hori and Cumrun Vafa.
\newblock Mirror symmetry.
\newblock {\em arXiv preprint hep-th/0002222}, 2000.

\bibitem[Iri09]{iritani}
Hiroshi Iritani.
\newblock An integral structure in quantum cohomology and mirror symmetry for
  toric orbifolds.
\newblock {\em Adv. Math.}, 222(3):1016--1079, 2009.

\bibitem[LS16]{LamSpeyerI}
Thomas Lam and David Speyer.
\newblock
Cohomology of cluster varieties. I. Locally acyclic case.
\newblock
{\em arXiv:1604.06843v1 [math.AG]}, 2016
\newblock


\bibitem[Lus83]{lusztig}
George Lusztig.
\newblock Singularities, character formulas, and a {$q$}-analog of weight
  multiplicities.
\newblock In {\em Analysis and topology on singular spaces, {II}, {III}
  ({L}uminy, 1981)}, volume 101 of {\em Ast\'erisque}, pages 208--229. Soc.
  Math. France, Paris, 1983.
  
\bibitem[Lus94]{lusztig:TP}
George Lusztig.
\newblock Total positivity in reductive groups. 
\newblock In {\em Lie theory and geometry}, volume 123,  
{\em Progr. Math.}, pages 531–-568, Birkh\"auser Boston, 1994. 

\bibitem[MR13]{MR}
R.~{M}arsh and K.~{R}ietsch.
\newblock {The $B$-model connection and $T$-equivariant mirror symmetry for
  Grassmannians}.
\newblock arXiv:1307.1085v2 [math.AG].

\bibitem[MV07]{MV}
I.~Mirkovi{\'c} and K.~Vilonen.
\newblock Geometric {L}anglands duality and representations of algebraic groups
  over commutative rings.
\newblock {\em Ann. of Math. (2)}, 166(1):95--143, 2007.

\bibitem[Pet97]{peterson}
D.~Peterson.
\newblock {Quantum cohomology of $G/P$}.
\newblock {Lecture Course, MIT, Spring Term}, 1997.

\bibitem[PR13]{PeRi}
C.~Pech and K.~Rietsch.
\newblock {A Landau-Ginzburg model for Lagrangian Grassmannians, Langlands
  duality and relations in quantum cohomology}.
\newblock arXiv:1304.4958, 2013.

\bibitem[PRW16]{PRW:quadrics}
C. Pech, K. Rietsch, and L. Williams.
\newblock On {L}andau--{G}inzburg models for quadrics and flat sections
              of {D}ubrovin connections.
\newblock {\em Adv. Math.}, 300:275--319, 2016.

\bibitem[Prz07]{przyjalkowski}
V.~V. Przhiyalkovski{\u\i}.
\newblock Gromov-{W}itten invariants of {F}ano threefolds of genera 6 and 8.
\newblock {\em Mat. Sb.}, 198(3):145--158, 2007.

\bibitem[Rie08]{rietsch}
Konstanze Rietsch.
\newblock A mirror symmetric construction of {$qH^\ast_T(G/P)_{(q)}$}.
\newblock {\em Adv. Math.}, 217(6):2401--2442, 2008.

\bibitem[Sab99]{sabbah}
Claude Sabbah.
\newblock Hypergeometric period for a tame polynomial.
\newblock {\em C. R. Acad. Sci. Paris S\'er. I Math.}, 328(7):603--608, 1999.

\end{thebibliography}
\end{document}